\documentclass[12pt]{article}

\usepackage{amsmath}
\usepackage{amssymb}
\usepackage{amscd}
\usepackage{bm}
\usepackage[all]{xy}

\usepackage{graphicx}

\usepackage{bbm}
\usepackage{latexsym}
\usepackage[all]{xy}
\usepackage{makeidx}
\usepackage{tocloft}
\usepackage{fancyhdr}

\usepackage{ifthen}

\newcommand{\lsection}[2][""]{%
    \ifthenelse{\equal{#1}{""}}{%
        \section{#2}
    }{%
        \renewcommand{\sectionmark}[1]{\markright{\thesection.\ \MakeUppercase{#1}}}
        \section{#2}
        \renewcommand{\sectionmark}[1]{\markright{\thesection.\ \MakeUppercase{##1}}}
    }
}

\newcommand{\lchapter}[2][""]{%
    \ifthenelse{\equal{#1}{""}}{%
        \chapter{#2}
    }{%
        \renewcommand{\chaptermark}[1]{\markboth{\MakeUppercase{\chaptername\ \thechapter.\ #1}}{}}
        \chapter{#2}
        \renewcommand{\chaptermark}[1]{\markboth{\MakeUppercase{\chaptername\ \thechapter.\ ##1}}{}}
    }
}

\def\R{\mathbb R}
\def\C{\mathbb C}

\def\P{\mathbb P}

\def\iff{if and only if}
\def\mfd{manifold}
\def\fcn{function}

\def\h{holomorphic}

\def\r{respectively}
\def\st{such that}
\def\(#1_#2){(#1_1,#1_2,\dots,#1_#2)}
\def\p #1_#2{#1_1#1_2\dots#1_#2}
\def\s#1_#2{#1_1+#1_2+\dots+#1_#2}
\def\wrt{with respect to}

\def\iso{isomorphism}
\def\ra{\rightarrow}
\def\lra{\longrightarrow}

\def\op{\operatorname}

\def\vb{vector bundle}

\def\ssl{\smallsmile}
\def\bp{\bar\partial}

\def\ssm{\hspace{-.5mm}\smallsetminus\hspace{-.5mm}}

\def\nbd{neighborhood}

\def\F{\mathcal F}

\def\U{\mathcal U}

\def\O{\mathcal O}

\def\cS{\mathcal S}
\def\cE{\mathcal E}
\def\cK{\mathcal K}

\def\a{\alpha}

\def\g{\gamma}
\def\na{\nabla}
\def\o{\omega}
\def\l{\lambda}
\def\t{\theta}
\def\z{\zeta}
\def\s{\sigma}

\def\vt{\vartheta}

\def\de{\partial}
\def\dbar{\bar\partial}

\def\dbarde{\bar\partial\partial}
\def\dbard{\bar\partial\partial}

\def\moins{\ssm}

\def\vv {\vskip.2cm}

\newtheorem{theorem}{Theorem}[section]
\newtheorem{lemma}[theorem]{Lemma}
\newtheorem{corollary}[theorem]{Corollary}
\newtheorem{proposition}[theorem]{Proposition}

\newtheorem{exa}[theorem]{Example}
\newenvironment{example}{\begin{exa} \em}{\end{exa}}
\newtheorem{exas}[theorem]{Examples}

\newtheorem{prope}[theorem]{Property}

\newtheorem{defini}[theorem]{Definition}
\newenvironment{definition}{\begin{defini} \em}{\end{defini}}

\newtheorem{rema}[theorem]{Remark}
\newenvironment{remark}{\begin{rema} \em}{\end{rema}}

\newenvironment{equationth}{\stepcounter{theorem}\begin{equation}}{\end{equation}}

\newenvironment{proof}{{\noindent \sc Proof: } }{\mbox{ }\hfill$\Box$  
                        \vspace{1.5ex} \par}

\title {\bf  \Large {Localization of Bott-Chern classes and\\ Hermitian residues}}

\author{Maur\'icio Corr\^ea Jr\thanks{Supported by  CAPES, CNPq and Fapesp}  and Tatsuo Suwa\thanks{Supported by JSPS grants no. 24540060 and no. 16K05116}
}

\date{}

\begin{document}

\pagestyle{plain}

\maketitle

\paragraph{Abstract\,:} 
We develop a theory of \v{C}ech-Bott-Chern cohomology and  in this context we naturally come up with the relative Bott-Chern cohomology. In fact Bott-Chern cohomology has two relatives and they all arise from a single complex.
Thus we study these three cohomologies in a unified way and obtain  a long exact sequence involving the three. We then study the localization problem
of characteristic classes in the relative Bott-Chern cohomology. For this we define the cup product and integration  in our framework and we discuss local and global duality morphisms. After reviewing some materials on connections, we give a vanishing theorem relevant to our localization. With these,  we prove a residue theorem for 
 vector bundles admitting a Hermitian connection compatible with an action of the non-singular part of a singular distribution. As a typical case, we discuss the  action of a distribution on the normal bundle of an invariant sub\mfd\ (the so-called
 Camacho-Sad action) and give a specific  example.
\bigskip

\vv


\lsection{Introduction}

\v{C}ech-de~Rham cohomology, particularly the relative version, together with its integration theory has been
extensively used in the localization problem of characteristic classes. It started with the residue problem of singular \h\ foliations
(cf. \cite{Leh3}, \cite{LS}, \cite{Sw1}) and the theory was then transferred to the fixed point theory of discrete dynamics
(cf. \cite{ABT}). The philosophy behind is rather simple. Namely, once we have some kind of vanishing theorem on the 
non-singular part, certain characteristic classes are localized at the set of singular points and the localization gives rise to residues and the residue theorem via the Alexander duality. The relative easiness of  computing the residues  is another advantage of this method.
The idea and the techniques turned out to be effective in many other problems,
 characteristic classes of singular varieties and localized intersection theory, to name a few (cf. \cite{BSS}, \cite{Su7}).
A similar theory is developed for the Dolbeault complex, the relevant characteristic classes in this case being the Atiyah classes
(cf. \cite{ABST}, \cite{Sw2}, \cite{Sw3}, \cite{Sw4}).

In this paper we study the localization problem in Bott-Chern cohomology. The cohomology, which was introduced in \cite{BC},
refines both de~Rham and Dolbeault cohomologies. It is a powerful tool in the study of non-K\"ahler \mfd s (cf.
\cite{ADT}, \cite{AT} and references therin) and is also related to arithmetic characteristic classes (cf. \cite{BKK}, \cite{GS1}).
We first develop a theory of \v{C}ech-Bott-Chern cohomology and the relative Bott-Chern cohomology naturally arises in this context. In fact  Bott-Chern cohomology has some relatives, i.e., Aeppli cohomology
and one more, and they all come from a single complex (cf. (\ref{mixed}) below). 
This viewpoint allows us to deal with the three types of cohomologies in a unified way and leads to a long exact sequence
involving the three   (Theorem \ref{exactmixed}). As to the local and global dualities, we are mainly interested in the
one between Bott-Chern and Aeppli. For this  the cup product is defined on the cochain level and the integration is defined via \v{C}ech-de~Rham cohomology.

In the de~Rham case, the usual Chern-Weil theory of characteristic classes fits nicely into the 
theory with some modifications
and, in the Dolbeault case, a similar strategy works by considering connections of type $(1,0)$.
 In the Bott-Chern case we impose one more condition on the connections, i.e.,
we require  the connections to be Hermitian. We then define characteristic classes in the \v{C}ech-Bott-Chern cohomology using the
Bott-Chern 
form. The relevant vanishing theorem in our case is the one proved in \cite{ABST} together with
the Hermitian condition (Corollary \ref{vanherm}).
With these at hand,  we have a general residue theorem for a \h\ \vb\ admitting a Hermitian connection compatible
with an action of the non-singular part of a singular distribution (Theorem \ref{residueth}).
\vv

Here is an outline of the paper.  In Section \ref{secBC}, we recall the three types of cohomologies, including Bott-Chern and Aeppli, by considering the complex (\ref{mixed}). We then recall, in Section \ref{relBC},  the theory of 
\v Cech cohomology of a complex of sheaves of differential forms and  we discuss its relative version. These are
applied to the three cases and we have the relative Bott-Chern, the relative Aeppli cohomologies and one more. As mentioned above, we  give a long
exact sequence  relating these three cohomologies  (Theorem \ref{exactmixed}). We also dicuss the relation with the Dolbeault case (Theorem \ref{threlD}). In order to state local and global duality morphisms, we discuss, in Section 4, the cup product in this cohomology theory, in particular the one between  Bott-Chern and Aeppli. Then we discuss the integration theory mainly on Aeppli cohomology.
 In Section \ref{secchBC},
we define the characteristic classes of a \h\ \vb\ in \v{C}ech-Bott-Chern cohomology. The vanishing theorem we need is given in Section \ref{ABvanish} (Corollary \ref{vanherm}).  It arises from an action of a non-singular distribution on a \vb\ and, as an
example, we present the action on the normal bundle of an invariant sub\mfd\ of a distribution, the so-called  Camacho-Sad action.
 In Section \ref{secloc}, we discuss localization theory of characteristic classes of \vb s with actions of singular distributions and the associated Hermitian residues (Theorem~\ref{residueth}). Finally we give an example in Section \ref{secex}.
 It concerns with a Hopf surface with a fibration of elliptic curves on the projective line. 
We show that the  Bott-Chern  class of the pull-back of the hyperplane bundle is localized at 
one of the fibers.
\vv

\noindent
{\em Acknowledgment}\,:
The first named author is  grateful to Hokkaido University for hospitality. 
The both authors would like to thank the referees for precious comments, in particular on Bott-Chern forms, which improved the presentation
of the paper greatly.

\lsection{Bott-Chern cohomology and its companions}\label{secBC}

We list \cite{ADT}, \cite{AT} and \cite{Sch} as references for this section.

Let $M$ denote a complex manifold of dimension $n$. Also let $\cE^{r}$ and $\cE^{p,q}$ denote the sheaves of 
$C^{\infty}$ differential $r$-forms and $(p,q)$-forms on $M$.
For a sheaf $\cS$ and an open set $U$ in $M$, we denote by $\cS(U)$  the set of sections on $U$.

We have the decomposition
\[
\cE^{r}=\bigoplus_{p+q=r}\cE^{p,q}
\]
and accordingly the exterior derivative $d:\cE^{r}\ra \cE^{r+1}$ has the
decomposition $d=\de+\dbar$\,:
\[
\de:\cE^{p,q}\lra \cE^{p+1,q},\qquad \dbar:\cE^{p,q}\lra \cE^{p,q+1}.
\]
Recall that the $r$-th de~Rham cohomology $H^r_d(M)$ of $M$ is the $r$-th cohomology of the complex 
$(\cE^{\bullet}(M),d)$ and the Dolbeault cohomology $H^{p,q}_{\bp}(M)$ of type $(p,q)$ of $M$ is the $q$-th cohomology of the
complex $(\cE^{p,\bullet}(M),\bp)$.

We set 
\[
\cE^{(p,q)+1}=\cE^{p+1,q}\oplus \cE^{p,q+1}
\]
and consider the complex 
\begin{equationth}\label{mixed}
\cdots\overset{d}\lra \cE^{(p-2,q-2)+1}\overset{\bp+\de}\lra \cE^{p-1,q-1}\overset\dbarde\lra \cE^{p,q}\overset{d}\lra \cE^{(p,q)+1}\overset{\bp+\de}\lra\cE^{p+1,q+1}\overset\dbarde\lra\cdots.
\end{equationth}

\paragraph{Bott-Chern cohomology\,:} The {\em Bott-Chern cohomology} $H^{p,q}_{\rm BC}(M)$ of type $(p,q)$ of
$M$ is the cohomology of the  complex
\[
\cE^{p-1,q-1}(M)\overset\dbarde\lra \cE^{p,q}(M)\overset{d}\lra \cE^{(p,q)+1}(M).
\]

We have  canonical morphisms

\begin{equationth}\label{nathomos}
H^{p,q}_{\rm BC}(M)\lra H^{p,q}_{\dbar}(M),\qquad \bigoplus_{p+q=r}H^{p,q}_{\rm BC}(M)\lra H^{r}_d(M).
\end{equationth}
If $M$ is compact K\"ahler, for instance, the above morphisms are \iso s.

\paragraph{Aeppli cohomology\,:}
The {\em Aeppli cohomology} $H^{p,q}_{\rm A}(M)$ of type $(p,q)$ of
$M$ is the cohomology of the complex
\[
\cE^{(p-1,q-1)+1}(M)\overset{\bp+\de}\lra \cE^{p,q}(M)\overset{\dbarde}\lra \cE^{p+1,q+1}(M).
\]
In particular, $H^{n,n}_{\rm A}(M)=H^{2n}_{d}(M)$.
We have  canonical morphisms

\begin{equationth}\label{nathomosa}
H^{p,q}_{\dbar}(M)\lra H^{p,q}_{\rm A}(M),\qquad H^{r}_d(M)\lra  \bigoplus_{p+q=r}H^{p,q}_{\rm A}(M).
\end{equationth}

Note  that the differential $\de:\cE^{p,q}(M)\ra \cE^{p+1,q}(M)$ induces a morphism
\begin{equationth}\label{ad}
\de:H^{p,q}_{\rm A}(M)\lra H^{p+1,q}_{\bp}(M).
 \end{equationth}
 
 \paragraph{The third cohomology\,:} We have one more cohomology, i.e., the cohomology of the complex
 \[
 \cE^{p,q}(M)\overset{d}\lra \cE^{(p,q)+1}(M)\overset{\bp+\de}\lra \cE^{p+1,q+1}(M),
 \]
which will be  denoted by $H^{(p,q)+1}(M)$.

\lsection{Relative Bott-Chern cohomology}\label{relBC}

\subsection{Relative cohomology of a complex of sheaves}\label{subCC}

A general theory of relative cohomology for complexes of sheaves is developed in \cite{Su11}. Here we recall
some of the relevant materials.

Suppose we have a complex $(\cK^{\bullet},d_{\cK})$ of sheaves 
of differential forms 
on $M$. In fact what we have in mind here are the  de~Rham and Dolbeault complexes and the one as in (\ref{mixed}).
We denote by $H^{r}_{d_{\cK}}(M)$ the $r$-th cohomology of $(\cK^{\bullet}(M),d_{\cK})$.

\paragraph{\v{C}ech cohomology of a complex\,:}
Let $\U=\{U_0,U_1\}$ be an open covering of $M$ and set
$U_{01}=U_0\cap U_1$.  We set
\[
\cK^{r}(\U)=\cK^{r}(U_0)\oplus \cK^{r}(U_1)\oplus \cK^{r-1}(U_{01}).
\]
Thus an element $\s$ in $\cK^{r}(\U)$ is given by a
triple $\s=(\s_0, \s_1, \s_{01})$ with $\s_i\in \cK^{r}(U_i)$, $i=0,1$, and $\s_{01}\in \cK^{r-1}(U_{01})$. 
We define the differential
\[
D_{\cK}=D_{\cK}^{r}:\cK^{r}(\U)\lra \cK^{r+1}(\U)\quad\text{by}\ \ D_{\cK}\s =(d_{\cK}\s_0, d_{\cK}\s_1, \s_{1}-\s_{0}-d_{\cK}\s_{01}).
\]
Then we see that $D_{\cK}\circ D_{\cK}=0$.

\begin{definition} The $r$-th  cohomology $H^{r}_{D_{\cK}}(\U)$ of $\cK^{\bullet}$ on $\U$ is the $r$-th cohomology of the complex
$(\cK^{\bullet}(\U),D_{\cK})$\,:
\[
H^{r}_{D_{\cK}}(\U)=\op{Ker}D^{r}_{\cK}/\op{Im}D^{r-1}_{\cK}.
\]
\end{definition}

If we set $Z^{r}_{d_{\cK}}(\U)=\op{Ker}D^{r}_{\cK}$ and $B^{r}_{d_{\cK}}(\U)=\op{Im}D^{r-1}_{\cK}$, by definition
\begin{equationth}\label{CKcocycob}
\begin{aligned}
Z^{r}_{d_{\cK}}(\U)&=\{\,\s\mid d_{\cK}\s_0=d_{\cK}\s_1=0,
\  \ \s_1- \s_0-d_{\cK}\s_{01}=0\,  \},\\
B^{r}_{d_{\cK}}(\U)&=\{\,\s\mid\exists
\tau,  \
\s_0=d_{\cK}\tau_{0},\ 
\s_1=d_{\cK}\tau_{1}, \  \ \s_{01}=\tau_1- \tau_0-d_{\cK}\tau_{01}\,\},
\end{aligned}
\end{equationth}
where $\s=(\s_0,\s_1,\s_{01})
\in \cK^{r}(\U)$ and $\tau=(\tau_0,\tau_1,\tau_{01})
\in \cK^{r-1}(\U)$.
We may somewhat simplify the expression of $B^{r}_{d_{\cK}}(\U)$, i.e., if we set
\[
B=\{\,\s\in \cK^{r}(\U)\mid\exists
(\tau_{0},\tau_{1}),  \
\s_0=d_{\cK}\tau_{0},\ 
\s_1=d_{\cK}\tau_{1}, \  \ \s_{01}=\tau_1- \tau_0\,\},
\]
where $(\tau_{0},\tau_{1})\in \cK^{r-1}(U_{0})\oplus\cK^{r-1}(U_{1})$, we have\,:

\begin{lemma}\label{lemred} $B^{r}_{d_{\cK}}(\U)=B$.
\end{lemma}

\begin{proof} Clearly we have $B\subset B^{r}_{d_{\cK}}(\U)$. Suppose $\s\in B^{r}_{d_{\cK}}(\U)$. Then
$\s_0=d_{\cK}\tau_{0},\ 
\s_1=d_{\cK}\tau_{1}, \  \ \s_{01}=\tau_1- \tau_0-d_{\cK}\tau_{01}$. Letting $\{\rho_{0},\rho_{1}\}$ be a partition of
unity subordinate to $\U$, we set
\[
\tau_{0}'=\tau_{0}+d_{\cK}(\rho_{1}\tau_{01}),\qquad \tau_{1}'=\tau_{1}-d_{\cK}(\rho_{0}\tau_{01}).
\]
Then $\s_0=d_{\cK}\tau_{0}',\ 
\s_1=d_{\cK}\tau_{1}', \  \ \s_{01}=\tau_1'- \tau_0'$.
\end{proof}

\begin{theorem}\label{isogen}
The map $\cK^{r}(M) \lra \cK^{r}(\U)$ given by
$\o \mapsto  \tilde\o=(\o, \o, 0)$ is compatible with the differentials and induces an
isomorphism
\[
\a:H^{r}_{d_{\cK}}(M) \overset{\sim}\lra H^{r}_{D_{\cK}}(\U).
\]
\end{theorem}

\begin{proof} The first part follows from the definitions so that $\a$ is  well-defined.
We show that $\a$ is surjective and
injective. To see the surjectivity, given $\s=(\s_0,\s_1,\s_{01})\in
Z^{r}_{d_{\cK}}(\U)$. We take  a partition of
unity $\{\rho_0,\rho_1\}$ subordinated to  $\U$  and
set
\[
\o=\begin{cases}\s_{0}+d_{\cK} (\rho_{1}\s_{01})\qquad\text{in}\ \ U_{0},\\
\s_{1}-d_{\cK} (\rho_{0}\s_{01})\qquad\text{in}\ \ U_{1}.
\end{cases}
\]
Then we see that it is a cocycle in $\cK^{r}(M)$. 
We claim that $\tilde\o- \s=(\o- \s_0,\o-
\s_1, -\s_{01} )$ is in $B^{r}_{d_{\cK}}(\U).$ Indeed,
setting $\tau_0=\rho_1\s_{01}$,  $\tau_1=-\rho_0\s_{01}$,
we  see that 
\[
\tilde\o- \s=D_{\cK}(\tau_{0},\tau_{1},0),
\]
which shows that $\widetilde{\o}- \s \in
B^{r}_{d_{\cK}}(\U)$ and thus  $\a$ is surjective.

To see the injectivity of $\a$, given
$\o\in\cK^{r}(M)$ with $d_{\cK}\o=0$ such that $[\widetilde{\o}]=0$, we claim
that there exist $\t\in \cK^{r-1}(M)$ such that $\o=d_{\cK}\t$. 
The condition $[\widetilde{\o}]=0$ means that there exists
$(\tau_0,\tau_1,0)$ such that (cf. Lemma \ref{lemred})
\[
(\o,\o,0)=D_{\cK}(\tau_0,\tau_1,0)=(d_{\cK}\tau_{0},d_{\cK}\tau_{1},\tau_{1}-\tau_{0}).
\]
Thus $\t=\tau_{0}=\tau_{1}$ is a global form with
$\o=d_{\cK} \t$.
\end{proof}

\begin{remark} The cohomology $H^{r}_{D_{\cK}}(\U)$ as defined above is the hypercohomology of the complex
$\cK^{\bullet}$ on $\U$. It can be defined for an arbitrary covering of $M$ and Theorem \ref{isogen} may be
proved in the general case, for instance by a spectral sequence argument
(cf. \cite{GH}, \cite{Sw1}, \cite{Sw2},\cite{Su11}).
\end{remark}

\paragraph{Relative cohomology of a complex\,:}

Let $S$ be a closed set in $M$. Setting $U_0=M\moins S$ and $U_1$ a \nbd\ of $S$, we consider the covering
$\U=\{U_0,U_1\}$ of $M$.
We set 
\[
\cK^{r}(\U,U_{0})=\{\,\s\in \cK^{r}(\U)\mid \s_{0}=0\,\}=\cK^{r}(U_1)\oplus \cK^{r-1}(U_{01}).
\]
Thus an element $\s$ in $\cK^{r}(\U,U_{0})$ is given by a
pair $\s=(\s_1, \s_{01})$ with $\s_1\in \cK^{r}(U_1)$ and $\s_{01}\in \cK^{r-1}(U_{01})$. 
Clearly $(\cK^{\bullet}(\U,U_{0}),D_{\cK})$ is a subcomplex of $(\cK^{\bullet}(\U),D_{\cK})$ with
$D_{\cK}:\cK^{r}(\U,U_{0})\ra \cK^{r+1}(\U,U_{0})$ given by $D_{\cK}\s=(d_{\cK}\s_{1},\s_{1}-d_{\cK}\s_{01})$.
Sometimes $(\s_1, \s_{01})$ is denoted by $(0,\s_1, \s_{01})$ to remember that it is an element in $\cK^{r}(\U)$.

\begin{definition} The $r$-th relative cohomology $H^{r}_{D_{K}}(\U,U_{0})$ of $\cK$ is the $r$-th cohomology of the complex
$(\cK^{\bullet}(\U,U_{0}),D_{\cK})$.
\end{definition}

Note that, in the relative case, we may not simplify the expression of the coboundaries as in Lemma \ref{lemred}.

From the  exact sequence of complexes
\[
0\lra \cK^{\bullet}(\U,U_{0})\lra \cK^{\bullet}(\U)\lra \cK^{\bullet}(U_{0})\lra 0,
\]
we have\,:

\begin{theorem}\label{thlongex}There is an exact sequence
\[
\cdots\lra H^{r-1}_{d_{\cK}}(U_{0})\overset{\delta^{*}}\lra H^{r}_{D_{\cK}}(\U,U_{0})\overset{j^{*}}\lra H^{r}_{D_{\cK}}(\U)\overset{i^{*}}\lra H^{r}_{d_{\cK}}(U_{0})\lra\cdots,
\]
where $\delta^{*}[\t]=[(0,-\t)]$, $j^{*}[(\s_{1},\s_{01})]=[(0,\s_{1},\s_{01})]$ and 
$i^{*}[(\s_{0},\s_{1},\s_{01})]=[\s_{0}]$.
\end{theorem}

By Theorem \ref{isogen} and the five lemma we have\,:

\begin{corollary}\label{corind} The cohomology $H^{r}_{D_{\cK}}(\U,U_{0})$ is determined uniquely up to canonical \iso s,
independently of the choice of $U_{1}$.
\end{corollary}

In view of the above, we may denote $H^{r}_{D_{\cK}}(\U,U_{0})$ by $H^{r}_{d_{\cK}}(M,M\ssm S)$. Then we have\,:
\begin{corollary}[Excision] For any open set $U$ containing $S$, there is a canonical \iso
\[
H^{r}_{d_{\cK}}(M,M\ssm S)\simeq H^{r}_{d_{\cK}}(U,U\ssm S).
\]
\end{corollary}

\begin{example}\label{exdRD} {\bf 1.} In the case $\cK^{\bullet}=\cE^{\bullet}$ and  $d_{\cK}=d$, we write  $D_{\cK}=D$. The cohomologies defined as above are the \v{C}ech-de~Rham cohomology $H^{r}_{D}(\U)$ and the relative de~Rham cohomology $H^{r}_{D}(\U,U_{0})$ (cf. \cite{BT}, \cite{Sw1}). Thus $H^{r}_{D}(\U)$ is the cohomology
of the complex $(\cE^{\bullet}(\U), D)$, where
\[
\cE^{r}(\U)=\cE^{r}(U_{0})\oplus \cE^{r}(U_{1})\oplus \cE^{r-1}(U_{01})
\]
and $D:\cE^{r}(\U)\ra \cE^{r+1}(\U)$ is given by
$D(\z_{0},\z_{1},\z_{01})=(d\z_{0},d\z_{1},\z_{1}-\z_{0}-d\z_{01})$.

By Theorem~\ref{isogen}, there is a canonical \iso
\begin{equationth}\label{isodr}
H^{r}_{d}(M)\simeq H^{r}_{D}(\U).
\end{equationth}

The relative de~Rham cohomology
$H^{r}_{D}(\U,U_{0})$ is the cohomology of the subcomplex $(\cE^{\bullet}(\U,U_{0}),D)$ of 
$(\cE^{\bullet}(\U),D)$, where
\[
\cE^{r}(\U,U_{0})=\cE^{r}(U_{1})\oplus\cE^{r-1}(U_{01}).
\]
\vv

\noindent
{\bf 2.}
 In the case $\cK^{\bullet}=\cE^{p,\bullet}$ and  $d_{\cK}=\bp$, we write  $D_{\cK}=\bar\vt$. The cohomologies as defined above are  the \v{C}ech-Dolbeault  cohomology $H^{p,q}_{\bar\vt}(\U)$ and the relative Dolbeault  cohomology $H^{p,q}_{\bar\vt}(\U,U_{0})$ (cf. \cite{Sw2}, where the operator $\bar\vt$ is denoted by $\bar D$, and \cite{Sw4}). 
 Thus $H^{p,q}_{\bar\vt}(\U)$ is the cohomology
of the complex $(\cE^{p,\bullet}(\U), \bar\vt)$, where
\[
\cE^{p,q}(\U)=\cE^{p,q}(U_{0})\oplus \cE^{p,q}(U_{1})\oplus \cE^{p,q-1}(U_{01})
\]
and $\bar\vt:\cE^{p,q}(\U)\ra \cE^{p,q+1}(\U)$ is given by
$\bar\vt(\z_{0},\z_{1},\z_{01})=(\bp\z_{0},\bp\z_{1},\z_{1}-\z_{0}-\bp\z_{01})$.

By Theorem \ref{isogen}, there is a canonical \iso
\begin{equationth}\label{isod}
H^{p,q}_{\bp}(M)\simeq H^{p,q}_{\bar\vt}(\U).
\end{equationth}

The relative Dolbeault cohomology
$H^{p,q}_{D}(\U,U_{0})$ is the cohomology of the subcomplex $(\cE^{p,\bullet}(\U,U_{0}),\bar\vt)$ of 
$(\cE^{p,\bullet}(\U),\bar\vt)$, where
\[
\cE^{p,q}(\U,U_{0})=\cE^{p,q}(U_{1})\oplus\cE^{p,q-1}(U_{01}).
\]
\end{example}
 
 \
 
In the sequel we examine each cohomology arising from (\ref{mixed}). For this we set
\begin{equationth}\label{cochaingen}
\begin{aligned}
\cE_{\rm BC}^{p,q}(\U)&=\cE^{p,q}(U_0)\oplus \cE^{p,q}(U_1)\oplus \cE^{p-1,q-1}(U_{01}),\\
\cE_{\rm A}^{p,q}(\U)&=\cE^{p,q}(U_0)\oplus \cE^{p,q}(U_1)\oplus \cE^{(p-1,q-1)+1}(U_{01}),\\
\cE^{(p,q)+1}(\U)&=\cE^{(p,q)+1}(U_0)\oplus \cE^{(p,q)+1}(U_1)\oplus \cE^{p,q}(U_{01})
\end{aligned}
\end{equationth}
and
\begin{equationth}\label{relcochaingen}
 \begin{aligned}
 \cE_{\rm BC}^{p,q}(\U,U_{0})&=\cE^{p,q}(U_1)\oplus \cE^{p-1,q-1}(U_{01}),\\
 \cE_{\rm A}^{p,q}(\U,U_{0})&=\cE^{p,q}(U_1)\oplus \cE^{(p-1,q-1)+1}(U_{01}),\\
 \cE^{(p,q)+1}(\U,U_{0})&=\cE^{(p,q)+1}(U_1)\oplus \cE^{p,q}(U_{01}).
 \end{aligned}
\end{equationth}

\subsection{Relative Bott-Chern cohomology} The relevant part of the complex is
\[
\cE^{(p-2,q-2)+1}\overset{\bp+\de}\lra \cE^{p-1,q-1}\overset\dbarde\lra \cE^{p,q}\overset{d}\lra \cE^{(p,q)+1}.
\]

\paragraph{\v{C}ech-Bott-Chern cohomology\,:}
The group of {\em \v{C}ech-Bott-Chern cochains}  of type $(p,q)$ is 
$\cE_{\rm BC}^{p,q}(\U)$.
Thus an element $\s$ in $\cE_{\rm BC}^{p,q}(\U)$ is given by a
triple $\s=(\s_0, \s_1, \s_{01})$.
The groups of  cocycles and  coboundaries are given by (cf. (\ref{CKcocycob}))
\begin{equationth}\label{CBCcocycob}
\begin{aligned}
Z^{p,q}_{\rm BC}(\U)&=\{\,\s\mid d\s_0=d\s_1=0,
\  \ \s_1- \s_0-\dbarde\s_{01}=0\,  \},\\
B^{p,q}_{\rm BC}(\U)&=\{\,\s\mid\exists
\tau,  \
\s_0=\dbarde\tau_{0},\ 
\s_1=\dbarde\tau_{1}, \  \ \s_{01}=
\tau_1- \tau_0-\bp\tau_{01}^{(1)}-\de\tau_{01}^{(2)}\,\},
\end{aligned}
\end{equationth}
where $\s\in \cE_{\rm BC}^{p,q}(\U)$ and 
$\tau=(\tau_0,\tau_1,\tau_{01}^{(1)},\tau_{01}^{(2)})\in
\cE^{p-1,q-1}_{\rm A}(\U)$.

\begin{definition} The {\em \v{C}ech-Bott-Chern cohomology }of type $(p,q)$ of $\U$ is defined by
\[
H^{p,q}_{\rm BC}(\U)=Z^{p,q}_{\rm BC}(\U)/B^{p,q}_{\rm BC}(\U).
\]
\end{definition}

By Theorem \ref{isogen}, 
the map $\cE^{p,q}(M) \ra\cE_{\rm BC}^{p,q}(\U)$ given by
$\o \mapsto  (\o, \o, 0)$ induces an
isomorphism
\begin{equationth}\label{isobc}
H^{p,q}_{\rm BC}(M) \overset{\sim}\lra
H^{p,q}_{\rm BC}(\U).
\end{equationth}
The inverse assigns to the class of $\s=(\s_{0},\s_{1},\s_{01})$ the class of the global  form
\[
\o= \rho_0 \s_0+ \rho_1 \s_1+\de \rho_0 \wedge
\dbar \s_{01}- \dbar\rho_0 \wedge
\de \s_{01}-
\dbarde \rho_0 \wedge \s_{01},
\]
where $(\rho_{0},\rho_{1})$ is a  partition of unity subordinate to $\U$.

By Lemma \ref{lemred}, we may simplify the expression of the coboundary group, i.e., we may drop
$\tau_{01}^{(1)}$ and $\tau_{01}^{(2)}$ in (\ref{CBCcocycob}).

\paragraph{Relative Bott-Chern cohomology\,:} The group of relative Bott-Chern cochains is
$\cE_{\rm BC}^{p,q}(\U,U_0)$.
The relative cocycle and coboundary groups are given by 
\[
\begin{aligned}
Z_{\rm BC}^{p,q}(\U,U_0)&=\{\,\s\mid d\s_1=0,\ 
\s_1-\dbarde \s_{01}=0\,\},\\
B_{\rm BC}^{p,q}(\U,U_0)&=\{\,\s\mid\exists\tau, \
\s_1=\dbarde\tau_{1}, \  \ \s_{01}=
\tau_1-\bp\tau_{01}^{(1)}-\de\tau_{01}^{(2)}\,\},
\end{aligned}
\]
where $\s\in \cE_{\rm BC}^{p,q}(\U,U_0)$ and $\tau=(\tau_{1},\tau_{01}^{(1)},\tau_{01}^{(2)})\in 
\cE^{p-1,q-1}_{\rm A}(\U,U_{0})$.

\begin{definition} The {\em relative Bott-Chern cohomology} is defined by
\[
H^{p,q}_{\rm BC}(\U,U_0)=Z_{\rm BC}^{p,q}(\U,U_0)/B_{\rm BC}^{p,q}(\U,U_0).
\]
\end{definition}

\subsection{Relative Aeppli cohomology}

The relevant part of the complex is
\[
\cE^{p-1,q-1}\overset{d}\lra \cE^{(p-1,q-1)+1}\overset{\bp+\de}\lra \cE^{p,q}\overset\dbarde\lra \cE^{p+1,q+1}.
\]

\paragraph{\v{C}ech-Aeppli cohomology\,:} Let $\U=\{U_{0},U_{1}\}$ be an open covering of $M$ as before.
The group of {\em \v{C}ech-Aeppli cochains} of type $(p,q)$ is $\cE_{\rm A}^{p,q}(\U)$. Thus an element $\xi$ in $ \cE_{\rm A}^{p,q}(\U)$ is given by a quadruple 
$(\xi_0,\xi_1,\xi_{01}^{(1)},\xi_{01}^{(2)})$.
The  cocycle and coboundary groups of type
$(p,q)$ are given by
\begin{equationth}\label{Acocycob}
\begin{aligned}
Z^{p,q}_{\rm A}(\U)&=\{\,\xi\mid\dbard\xi_0=
\dbard\xi_1=0, \  \ \xi_1- \xi_0-\bp\xi_{01}^{(1)}-\de\xi_{01}^{(2)}=0\,   \},\\
B^{p,q}_{\rm A}(\U)&=\{\,\xi\mid\exists \eta, \
\xi_i= \bp \eta_i^{(1)}+ \de \eta_i^{(2)}, 
\ \xi_{01}^{(1)}= \eta_{1}^{(1)}- \eta_{0}^{(1)}-\de\eta_{01}, \ \xi_{01}^{(2)}= \eta_{1}^{(2)}- \eta_{0}^{(2)}-\bp\eta_{01}\,\},
\end{aligned}
\end{equationth}
where $\xi\in \cE_{\rm A}^{p,q}(\U)$ and $\eta=(\eta_{0}^{(1)},\eta_{0}^{(2)},\eta_{1}^{(1)},\eta_{1}^{(2)},\eta_{01})\in \cE^{(p-1,q-1)+1}(\U)$, i.e.,
$\eta_{i}^{(1)}\in \cE^{p,q-1}(U_{i})$,
$\eta_{i}^{(2)}\in \cE^{p-1,q}(U_{i})$, $i=0,1$, and $\eta_{01}\in \cE^{p-1,q-1}(U_{01})$. 

\begin{definition} The
{\em \v{C}ech-Aeppli cohomology} of type $(p, q)$
of $\U$ is defined by
\[
H^{p,q}_{\rm A}(\U)=Z^{p,q}_{\rm A}(\U)/B^{p,q}_{\rm A}(\U).
\]
\end{definition}

By Theorem \ref{isogen}, the map $\cE^{p,q}(M) \ra  \cE_{\rm A}^{p,q}(\U)$ given by
$\o \mapsto  (\o, \o, 0,0)$ induces an
isomorphism
\begin{equationth}\label{isoa}
H^{p,q}_{\rm A}(M) \overset{\sim}\lra
H^{p,q}_{\rm A}(\U).
\end{equationth}
The inverse is given by assigning to the class of $\s$ the class of the global form
\[
\o= \rho_0 \xi_0+ \rho_1 \xi_1-\de \rho_0 \wedge
\xi_{01}^{(1)}- \dbar\rho_0 \wedge\xi_{01}^{(2)}.
\]

By Lemma \ref{lemred}, we may simplify the expression of the coboundary group, i.e., we may set $\eta_{01}=0$ in (\ref{Acocycob}).

We may also relax the coboundary condition.
Thus we set
\[
B'=\{\,\xi\mid \xi_i= \bp \eta_i^{(1)}+ \de \eta_i^{(2)}, 
\ \xi_{01}^{(k)}= \eta_{1}^{(k)}- \eta_{0}^{(k)}-\bp\eta_{01}^{(k1)}-\de\eta_{01}^{(k2)},\ i=0,1,\ k=1,2\,\},
\]
where $\xi\in \cE_{\rm A}^{p,q}(\U)$, $\eta_{i}^{(1)}\in \cE^{p,q-1}(U_{i})$,
$\eta_{i}^{(2)}\in \cE^{p-1,q}(U_{i})$, $i=0,1$,  $\eta_{01}^{(11)}\in \cE^{p,q-2}(U_{01})$, $\eta_{01}^{(12)}=\eta_{01}^{(21)}\in \cE^{p-1,q-1}(U_{01})$, $\eta_{01}^{(22)}\in \cE^{p-2,q}(U_{01})$.

\begin{lemma}\label{lemAe} We have $B_{A}^{p,q}(\U)=B'$.
\end{lemma}

\begin{proof} Clearly $B_{A}^{p,q}(\U)\subset B'$. Take an element $\xi\in B'$. Denoting by $\{\rho_{0},\rho_{1}\}$ a partition of unity subordinate to $\U$ we set
\[
\chi_{0}^{(k)}=\eta_{0}^{(k)}+\bp(\rho_{1}\eta_{01}^{(k1)})+\de(\rho_{1}\eta_{01}^{(k2)}),\qquad
\chi_{1}^{(k)}=\eta_{1}^{(k)}-\bp(\rho_{0}\eta_{01}^{(k1)})-\de(\rho_{0}\eta_{01}^{(k2)}).
\]
Then $\xi_i= \bp \chi_i^{(1)}+ \de \chi_i^{(2)}, 
\ \xi_{01}^{(k)}= \chi_{1}^{(k)}- \chi_{0}^{(k)},\ i=0,1,\ k=1,2$.
\end{proof}

The following is proved by applying Lemma \ref{lemred} for the \v{C}ech-de~Rham coboundaries or by
applying Lemma \ref{lemAe}\,:

\begin{proposition}\label{propCA} The morphism $\cE^{n,n}_{\rm A}(\U)\ra \cE^{2n}(\U)$ given by 
\[
(\xi_{0},\xi_{1},\xi_{01}^{(1)},\xi_{01}^{(2)})\mapsto (\xi_{0},\xi_{1},\xi_{01}^{(1)}+\xi_{01}^{(2)})
\]
 induces an \iso
\[
H^{n,n}_{\rm A}(\U)\overset\sim\lra H^{2n}_{D}(\U).
\]
\end{proposition}

The above \iso\ is compatible with the identity $H^{n,n}_{A}(M)=H^{2n}_{d}(M)$ via the \iso s
of Theorem \ref{isogen} for the Aeppli and de~Rham cases.

\paragraph{Relative Aeppli cohomology\,:}
Let $S$ be a closed set in $M$. Setting $U_0=M\moins S$ and $U_1$ a \nbd\ of $S$, we consider the covering
$\U=\{U_0,U_1\}$ of $M$. The group of relative Aeppli cochains is  
$\cE_{\rm A}^{p,q}(\U,U_0)$.
The  relative cocycle  and coboundary groups are given by  
\[
\begin{aligned}
Z^{p,q}_{\rm A}(\U,U_{0})&=\{\,\xi\mid
\dbard\xi_1=0, \  \ \xi_1-\dbar\xi_{01}^{(1)}-\de\xi_{01}^{(2)}=0\,   \},\\
B^{p,q}_{\rm A}(\U,U_{0})&=\{\,\xi\mid\exists\eta, \ \xi_1= \dbar \eta_1^{(1)}+ \de \eta_1^{(2)}, 
\ \xi_{01}^{(1)}= \eta_{1}^{(1)}-\de\eta_{01},\ \xi_{01}^{(2)}= \eta_{1}^{(2)}-\dbar\eta_{01}\,\},
\end{aligned}
\]
where $\xi\in \cE_{\rm A}^{p,q}(\U,U_{0})$ and $\eta=(\eta_{1}^{(1)},\eta_{1}^{(2)},\eta_{01})\in \cE^{(p-1,q-1)+1}(\U,U_{0})$.

\begin{definition} The {\em relative Aeppli cohomology} is defined by
\[
H^{p,q}_{\rm A}(\U,U_0)=Z^{p,q}_{\rm A}(\U,U_0)/B^{p,q}_{\rm A}(\U,U_0).
\]
\end{definition}

In the relative case, Lemma \ref{lemred} or \ref{lemAe} does not apply and we do not have an \iso\
as in Proposition \ref{propCA}. However we have\,:

\begin{proposition}\label{proprA} The morphism $\cE^{n,n}_{\rm A}(\U,U_{0})\ra \cE^{2n}(\U,U_{0})$ given by 
\[
(\xi_{1},\xi_{01}^{(1)},\xi_{01}^{(2)})\mapsto (\xi_{1},\xi_{01}^{(1)}+\xi_{01}^{(2)})
\]
 induces an epimorphism
\[
H^{n,n}_{\rm A}(\U,U_{0})\lra H^{2n}_{D}(\U,U_{0}).
\]
\end{proposition}

\subsection{An exact sequence}  

Considering 
\[
\cE^{p-1,q-1}\overset\dbarde\lra \cE^{p,q}\overset{d}\lra \cE^{(p,q)+1}\overset{\bp+\de}\lra\cE^{p+1,q+1},
\]
we may also define $H^{(p,q)+1}(\U)$ and $H^{(p,q)+1}(\U,U_{0})$ using the cochains $\cE^{(p,q)+1}(\U)$ and $\cE^{(p,q)+1}(\U,U_{0})$ (cf. (\ref{cochaingen}) and (\ref{relcochaingen})).

By focusing on the Bott-Chern cohomology in  Theorm \ref{thlongex}, we have\,:

\begin{theorem}\label{exactmixed} There is an exact sequence
\[
\begin{aligned}
\cdots \lra &H^{p-1,q-1}_{\rm A}(\U)\lra H^{p-1,q-1}_{\rm A}(U_{0})\overset{\delta^{*}}\lra H^{p,q}_{\rm BC}(\U,U_{0})\overset{j^{*}}\lra H^{p,q}_{\rm BC}(\U)\overset{i^{*}}\lra H^{p,q}_{\rm BC}(U_{0})\\
&\lra H^{(p,q)+1}(\U,U_{0})\lra H^{(p,q)+1}(\U)\lra H^{(p,q)+1}(U_{0})\lra H^{p+1,q+1}_{\rm A}(\U,U_{0})\lra\cdots.
\end{aligned}
\]
\end{theorem}

Now we discuss the relation with the  Dolbeault case (cf. Example \ref{exdRD}.\,2).
The following two propositions follow directly from the definition\,:

\begin{proposition}\label{propbcd} {\bf 1.} The map $\cE^{p,q}_{\rm BC}(\U)\ra \cE^{p,q}(\U)$ given by $(\s_{0},\s_{1},\s_{01})\mapsto (\s_{0},\s_{1},\de\s_{01})$ induces a morphism
\[
H^{p,q}_{\rm BC}(\U)\lra H^{p,q}_{\bar\vt}(\U),
\]
which is compatible with the first morphism in {\rm (\ref{nathomos})} via the \iso s   {\rm (\ref{isobc})} and {\rm (\ref{isod})}.
\vv

\noindent
{\bf 2.} The map $\cE^{p,q}_{\rm BC}(\U,U_{0})\ra \cE^{p,q}(\U,U_{0})$ given by $(\s_{1},\s_{01})\mapsto (\s_{1},\de\s_{01})$ induces a morphism
\[
H^{p,q}_{\rm BC}(\U,U_{0})\lra H^{p,q}_{\bar\vt}(\U,U_{0}).
\]
\end{proposition}

\begin{proposition}\label{propad} {\bf 1.}
The map $\cE^{p,q}_{\rm A}(\U)\ra \cE^{p+1,q}(\U)$ given by 
\[
(\xi_{0},\xi_{1},\xi^{(1)}_{01},\xi^{(2)}_{01})\mapsto (\de\xi_{0},\de\xi_{1},-\de\xi^{(1)}_{01})
\]
 induces a morphism
\[
\de:H^{p,q}_{\rm A}(\U)\lra H^{p+1,q}_{\bar\vt}(\U),
\]
which is compatible with the morphism  {\rm (\ref{ad})} via the \iso s   {\rm (\ref{isoa})} and {\rm (\ref{isod})}.
\vv

\noindent
{\bf 2.} The map $\cE^{p,q}_{\rm A}(\U,U_{0})\ra \cE^{p+1,q}(\U,U_{0})$ given by 
\[
(\xi_{1},\xi^{(1)}_{01},\xi^{(2)}_{01})\mapsto (\de\xi_{1},-\de\xi^{(1)}_{01})
\]
 induces a morphism
\[
\de:H^{p,q}_{\rm A}(\U,U_{0})\lra H^{p+1,q}_{\bar\vt}(\U,U_{0}).
\]
\end{proposition}

\begin{theorem}\label{threlD} We have the following commutative diagram with exact rows\,{\rm :} 
\[
\SelectTips{cm}{}
\xymatrix@R=.7cm
@C=.4cm{
H_{\rm A}^{p-1,q-1}(\U,U_{0})\ar[r]\ar[d]^-{\de}&H_{\rm A}^{p-1,q-1}(\U)\ar[r]\ar[d]^-{\de} &H_{\rm A}^{p-1,q-1}(U_0)\ar[r]\ar[d]^-{\de} & H_{\rm BC}^{p,q}(\U, U_0) \ar[d]\ar[r] &  H_{\rm BC}^{p,q} (\U)\ar[d]\ar[r] &  H_{\rm BC}^{p,q} (U_{0})\ar[d]
\\
H_{\bar\vt}^{p,q-1}(\U,U_{0})\ar[r]&H_{\bar\vt}^{p,q-1}(\U)\ar[r] & H_{\dbar}^{p,q-1}(U_0)\ar[r] &H_{\bar\vartheta}^{p,q}(\U, U_0)  \ar[r] &  H_{\bar\vartheta}^{p,q} (\U)\ar[r] &  H_{\dbar}^{p,q} (U_{0}).}
\]
\end{theorem}

In the above,  the first row is a part of the sequence in Theorem \ref{exactmixed} and the second row is the one in Theorem 
\ref{thlongex} for the Dolbeault complex (cf. Example \ref{exdRD}.\,2).

\lsection{Cup product,  integration and duality}\label{seccup}

Let $M$ be a complex \mfd\ of dimension $n$.

\subsection{Cup product}

Recall that the exterior product induces a bilinear map (cf. \cite{Sch})
\begin{equationth}\label{cupbca}
 H^{p,q}_{\rm BC}(M)\times H^{r,s}_{\rm A}(M)\overset\wedge\lra H^{p+r,q+s}_{\rm A}(M).
\end{equationth}


Now we try to find a product in \v{C}ech-Bott-Chern and \v{C}ech-Aeppli cohomologies that corresponds to
the above.
We define a cup product 
\[
\cE_{\rm BC}^{p,q}(\U)\times
 \cE_{\rm A}^{r,s}(\U)\overset\ssl\lra  \cE_{\rm A}^{p+r,q+s}(\U)
 \]
  by assigning to $(\s,\xi)$,
 $\s=(\s_0,\s_1,\s_{01})$, $\xi=(
\xi_0,\xi_1,\xi_{01}^{(1)},\xi_{01}^{(2)})$,
 the cochain $\s\ssl\xi$ given by

\begin{equationth}\label{cup}
\begin{aligned}
(\s\ssl\xi)_{i}&=\s_i\wedge \xi_i,\quad i=0,1,\\
(\s\ssl\xi)_{01}^{(1)}&=(-1)^{p+q}\s_0\wedge\xi_{01}^{(1)}+\de\s_{01}\wedge \xi_1 -\frac{1}{2}\de (\s_{01}\wedge \xi_1),\\
(\s\ssl\xi)_{01}^{(2)}&=(-1)^{p+q}\s_0\wedge\xi_{01}^{(2)}-\dbar\s_{01}\wedge \xi_1  +\frac{1}{2}\dbar (\s_{01}\wedge \xi_1).
\end{aligned}
\end{equationth}

\begin{proposition}\label{cupcbcca}
The above cup product induces a bilinear map
\[
 H_{\rm BC}^{p,q}(\U) \times
 H_{\rm A}^{r,s}(\U)\overset\ssl\lra  H_{\rm A}^{p+r,q+s}(\U).
\]
\end{proposition}

\begin{proof} We need to show that
\begin{enumerate}
\item[(1)] If $\s\in Z_{\rm BC}^{p,q}(\U)$ and $\xi\in Z_{\rm A}^{r,s}(\U)$, then $\s\ssl\xi\in Z_{\rm A}^{p+r,q+s}(\U)$.
\item[(2)] If $\s\in B_{\rm BC}^{p,q}(\U)$ and $\xi\in Z_{\rm A}^{r,s}(\U)$, then $\s\ssl\xi\in B_{\rm A}^{p+r,q+s}(\U)$.
\item[(3)] If $\s\in Z_{\rm BC}^{p,q}(\U)$ and $\xi\in B_{\rm A}^{r,s}(\U)$, then $\s\ssl\xi\in B_{\rm A}^{p+r,q+s}(\U)$.
\end{enumerate}

First, we note that, for a $(p,q)$-form $\o$ and an $(r,s)$-form $\t$,
\begin{equationth}\label{eqdbarde}
\dbarde(\o\wedge\t)=\dbarde\o\wedge\t+(-1)^{p+q+1}\de\o\wedge\bp\t+(-1)^{p+q}\bp\o\wedge\de\t+\o\wedge\dbarde\t.
\end{equationth}

The statement (1) follows from direct computations using 
(\ref{eqdbarde}).

To prove (2), take $\s\in B_{\rm BC}^{p,q}(\U)$ and $\xi\in Z_{\rm A}^{r,s}(\U)$. Thus $\s_{i}=\dbarde\tau_{i}$,
$i=0, 1$, and $\s_{01}=\tau_{1}-\tau_{0}$ (cf. Lemma \ref{lemred}). Also $\dbarde\xi_{i}=0$, $i=0,1$, and $\xi_{1}-\xi_{0}-\bp\xi_{01}^{(1)}
-\de\xi_{01}^{(2)}=0$. Then we may write
\[
(\s\ssl\xi)_{i}=\bp\a_{i}^{(1)}+\de\a_{i}^{(2)},
\]
where
\[
\a_{i}^{(1)}=-(-1)^{p+q}\tau_{i}\wedge\de\xi_{i}+\frac 1 2 \de(\tau_{i}\wedge\xi_{i}),\quad
\a_{i}^{(2)}=(-1)^{p+q}\tau_{i}\wedge\bp\xi_{i}-\frac 1 2 \bp(\tau_{i}\wedge\xi_{i}).
\]
We compute 
\[
(\s\ssl\xi)_{01}^{(1)}=\a_{1}^{(1)}-\a_{0}^{(1)}-\bp\a_{01}^{(11)}-\de\a_{01}^{(12)},\quad
(\s\ssl\xi)_{01}^{(2)}=\a_{1}^{(2)}-\a_{0}^{(2)}-\bp\a_{01}^{(21)}-\de\a_{01}^{(22)},
\]
where
\[
\begin{aligned}
\a_{01}^{(11)}&=
-(-1)^{p+q}\de\tau_{0}\wedge\xi_{01}^{(1)},\qquad
\a_{01}^{(12)}=\a_{01}^{(21)}=-\frac 12 \tau_{0}\wedge(\bp\xi_{01}^{(1)}-\de\xi_{01}^{(2)}),\\
\a_{01}^{(22)}&=
(-1)^{p+q}\bp\tau_{0}\wedge\xi_{01}^{(2)}.
\end{aligned}
\]
Thus $\s\ssl\xi\in B^{p+r,q+s}_{\rm A}(\U)$ (cf. Lemma \ref{lemAe}).

To prove (3), take $\s\in Z_{\rm BC}^{p,q}(\U)$ and $\xi\in B_{\rm A}^{r,s}(\U)$. Thus $d\s_{i}=0$, $i=0, 1$,
$\s_{1}-\s_{0}-\dbarde\s_{01}=0$ and $\xi_{i}=\bp\eta_{i}^{(1)}+\de\eta_{i}^{(2)}$, $\xi_{01}^{(k)}=\eta_{1}^{(k)}
-\eta_{0}^{(k)}$, $i=0,1$, $k=1,2$,  (cf. Lemma \ref{lemred}). Then we may write
\[
(\s\ssl\xi)_{i}=\bp\beta_{i}^{(1)}+\de\beta_{i}^{(2)},\qquad \beta_{i}^{(k)}=(-1)^{p+q}\s_{i}\wedge\eta_{i}^{(k)}.
\]
We compute
\[
(\s\ssl\xi)_{01}^{(1)}=\beta_{1}^{(1)}-\beta_{0}^{(1)}-\bp\beta_{01}^{(11)}-\de\beta_{01}^{(12)},\quad
(\s\ssl\xi)_{01}^{(2)}=\beta_{1}^{(2)}-\beta_{0}^{(2)}-\bp\beta_{01}^{(21)}-\de\beta_{01}^{(22)},
\]
where
\[
\begin{aligned}
\beta_{01}^{(11)}&=(-1)^{p+q}\de\s_{01}\wedge\eta_{1}^{(1)}, \qquad
\beta_{01}^{(12)}=\beta_{01}^{(21)}=\frac 1 2 \s_{01}\wedge(\bp\eta_{1}^{(1)}-\de\eta_{1}^{(2)}),\\
\beta_{01}^{(22)}&=-(-1)^{p+q}\bp\s_{01}\wedge\eta_{1}^{(2)}.
\end{aligned}
\]
Thus $\s\ssl\xi\in B^{p+r,q+s}_{\rm A}(\U)$.
\end{proof}

Note that the cup product in Proposition \ref{cupcbcca} is compatible with (\ref{cupbca}) via the \iso s
(\ref{isobc}) and (\ref{isoa}).

\begin{remark}\label{remcobBC} If we do not use  Lemma \ref{lemred} and write $\s_{01}=\tau_{1}-\tau_{0}-\bp\tau_{01}^{(1)}-\de\tau_{01}^{(2)}$ in the step (2) above, the terms $(\s\ssl\xi)_{i}$, $i=0, 1$, remain the same and we have
\begin{equationth}\label{prod}
\begin{aligned}
(\s\ssl\xi)_{01}^{(1)}&=\a_{1}^{(1)}-\a_{0}^{(1)}-\bp(\a_{01}^{(11)}+\gamma_{01}^{(11)})-\de(\a_{01}^{(12)}+\gamma_{01}^{(12)}),\\
(\s\ssl\xi)_{01}^{(2)}&=\a_{1}^{(2)}-\a_{0}^{(2)}-\bp(\a_{01}^{(21)}+\gamma_{01}^{(21)})-\de(\a_{01}^{(22)}+\gamma_{01}^{(22)}),
\end{aligned}
\end{equationth}
where $\a_{01}^{(kl)}$, $k, l=1, 2$, are as before and
\[
\begin{aligned}
\gamma_{01}^{(11)}&=-(-1)^{p+q}\tau_{01}^{(1)}\wedge\de\xi_{1},\qquad
\gamma_{01}^{(12)}=\gamma_{01}^{(21)}=\frac 12 (\bp\tau_{01}^{(1)}-\de\tau_{01}^{(2)})\wedge\xi_{1},\\
\gamma_{01}^{(22)}&=
(-1)^{p+q}\tau_{01}^{(2)}\wedge\bp\xi_{1}.
\end{aligned}
\]
Thus $\s\ssl\xi$ is still in $B^{p+r,q+s}_{\rm A}(\U)$ (cf. Lemma \ref{lemAe}).
This is necessary when we consider the relative case (cf. the proof of Proposition \ref{propcuprel} below).
\end{remark}

Now we consider the relative case. Thus let $S$ be a closed set in $M$. Let $U_{0}=M\ssm S$ and $U_{1}$ a \nbd\ of
$S$ and consider the  covering $\U=\{U_{0},U_{1}\}$ as before.
By setting $\s_{0}=0$ in  (\ref{cup}), we see that it induces 
 a cup product 
 \[
 \cE_{\rm BC}^{p,q}(\U,U_0)\times
 \cE^{r,s}(U_{1})\overset\ssl\lra  \cE_{\rm A}^{p+r,q+s}(\U,U_0)
 \]
 given by
 \begin{equationth}\label{cup2}
(\s,\xi)\mapsto (\s_1\wedge \xi_1,
\de\s_{01}\wedge \xi_1  -\frac{1}{2}\de (\s_{01}\wedge \xi_1),  -  \bp\s_{01}\wedge \xi_1 +\frac{1}{2}\bp (\s_{01}\wedge \xi_1) ).
\end{equationth}
It does not induce the cup product on the corresponding cohomologies, since Lemma~\ref{lemAe} does not apply in the relative case. However, for our purpose it suffices to
consider the case $p+r=q+s=n$ and in this case, we can make the receiving cohomology the relative de~Rham
cohomology (cf. Example \ref{exdRD}.\,1). 
Namely, if we consider the composition of 
\eqref{cup2} and $\cE^{n,n}_{A}(\U,U_{0})\ra \cE^{2n}(\U,U_{0})$ given by $(\xi_{1},\xi_{01}^{(1)},\xi_{01}^{(2)})\mapsto (\xi_{1},\xi_{01}^{(1)}+\xi_{01}^{(2)})$ (cf. Proposition \ref{proprA}), we have a
bilinear map
 \[
 \cE_{\rm BC}^{p,q}(\U,U_0)\times
 \cE^{n-p,n-q}(U_{1})\overset{\cdot}\lra  \cE^{2n}(\U,U_0)
 \]
 given by
 \begin{equationth}\label{cup3}
(\s,\xi_{1})\mapsto \s\cdot\xi_{1}=(\s_1\wedge \xi_1,
((\de-\bp)\s_{01})\wedge \xi_1  -\frac{1}{2}(\de-\bp) (\s_{01}\wedge \xi_1)).
\end{equationth}

\begin{proposition}\label{propcuprel} The above bilinear map induces a bilinear map
\[
H_{\rm BC}^{p,q}(\U,U_0)\times
H_{\rm A}^{n-p,n-q}(U_1)\overset{\cdot}\lra  H_{D}^{2n}(\U,U_0).
\]
\end{proposition}
\begin{proof} We need to show that
\begin{enumerate}
\item[(1)] If $\s\in Z_{\rm BC}^{p,q}(\U,U_{0})$ and $\xi_{1}\in Z_{\rm A}^{n-p,n-q}(U_{1})$, then $\s\cdot\xi_{1}\in Z^{2n}(\U,U_{0})$.
\item[(2)] If $\s\in B_{\rm BC}^{p,q}(\U,U_{0})$ and $\xi_{1}\in Z_{\rm A}^{n-p,n-q}(U_{1})$, then $\s\cdot\xi_{1}\in B^{2n}(\U,U_{0})$.
\item[(3)] If $\s\in Z_{\rm BC}^{p,q}(\U,U_{0})$ and $\xi_{1}\in B_{\rm A}^{n-p,n-q}(U_{1})$, then $\s\cdot\xi_{1}\in B^{2n}(\U,U_{0})$.
\end{enumerate}
In the above $Z_{\rm A}^{n-p,n-q}(U_{1})$ and $B_{\rm A}^{n-p,n-q}(U_{1})$ denote the Aeppli cocycles and
coboundaries on $U_{1}$ and $Z^{2n}(\U,U_{0})$ and $B^{2n}(\U,U_{0})$ the relative de~Rham 
cocycles and
coboundaries.

To prove (1), take $\s\in Z_{\rm BC}^{p,q}(\U,U_{0})$ and $\xi_{1}\in Z_{\rm A}^{n-p,n-q}(U_{1})$. 
Thus $d\s_{1}=0$, 
$\s_{1}-\dbarde\s_{01}=0$
and $\dbarde\xi_{1}=0$. We have $d(\s\cdot\xi_{1})_{1}=d(\s_{1}\wedge\xi_{1})=0$ by the degree reason.
By \eqref{eqdbarde} and the degree reason we compute $(\s\cdot\xi_{1})_{1}-d(\s\cdot\xi_{1})_{01}=0$.
Thus $\s\cdot\xi_{1}\in Z^{2n}(\U,U_{0})$.

To prove (2), take $\s\in B_{\rm BC}^{p,q}(\U,U_{0})$ and $\xi\in Z_{\rm A}^{n-p,n-q}(U_{1})$. Thus $\s_{1}=\dbarde\tau_{1}$ and $\s_{01}=\tau_{1}-\bp\tau_{01}^{(1)}-\de\tau_{01}^{(2)}$ (cf. Remark \ref{remcobBC}). Also $\dbarde\xi_{1}=0$.
Then we may write, again by the degree reason,
\[
(\s\cdot\xi_{1})_{1}=d(\a_{1}^{(1)}+\a_{1}^{(2)}),
\]
where  $\a_{1}^{(k)}$, $k=1,2$, are as in the proof of Proposition~\ref{cupcbcca}.
In \eqref{prod}, we have  $\a_{01}^{(kl)}=0$ as $\tau_{0}=0$. Also $\bp\gamma_{01}^{(11)}=d\gamma_{01}^{(11)}$ and $\de\gamma_{01}^{(22)}=d\gamma_{01}^{(22)}$ as $\gamma_{01}^{(11)}$ and 
$\gamma_{01}^{(22)}$ are $(n,n-2)$ and $(n-2,n)$-forms, \r. Thus we have
\[
(\s\cdot\xi_{1})_{01}=\a_{1}^{(1)}+\a_{1}^{(2)}-d(\gamma_{01}^{(11)}+\gamma_{01}
+\gamma_{01}^{(22)}),
\]
where $\gamma_{01}=\gamma_{01}^{(12)}=\gamma_{01}^{(21)}$.
This shows $\s\cdot\xi_{1}\in B^{2n}(\U,U_{0})$.

To prove (3), take $\s\in Z_{\rm BC}^{p,q}(\U,U_{0})$ and $\xi_{1}\in B_{\rm A}^{n-p,n-q}(U_{1})$. Thus $d\s_{1}=0$, 
$\s_{1}-\dbarde\s_{01}=0$ and $\xi_{1}=\bp\eta_{1}^{(1)}+\de\eta_{1}^{(2)}$. Then we have,
with $\beta_{1}^{(k)}$ and $\beta_{01}^{(kl)}$ as in the proof of Proposition~\ref{cupcbcca},
\[
(\s\cdot\xi_{1})_{1}=d(\beta_{1}^{(1)}+\beta_{1}^{(2)}),\qquad (\s\cdot\xi_{1})_{01}=\beta_{1}^{(1)}+
\beta_{1}^{(2)}-d(\beta_{01}^{(11)}+\beta_{01}+\beta_{01}^{(22)}),
\]
where $\beta_{01}=\beta_{01}^{(12)}=\beta_{01}^{(21)}$,
which shows $\s\cdot\xi_{1}\in B^{2n}(\U,U_{0})$.
\end{proof}

\subsection{Integration}

We may define integration on  \v{C}ech-Bott-Chern or  \v{C}ech-Aeppli cohomology by making use of
the integration theory on \v Cech-de~Rham cohomology.
The \v Cech-de~Rham cohomology and its integration theory may be developed  for an arbitrary covering of a $C^{\infty}$ \mfd. Here we 
briefly recall the theory
in our situation and refer to \cite{Le1} and \cite{Sw1} for the general case and details. 

Let $M$ be a complex \mfd\ of dimension $n$ and $\U=\{U_{0},U_{1}\}$ a covering of $M$ as before.
The \v Cech-de~Rham cohomology $H^{r}_{D}(\U)$ is defined as in Example \ref{exdRD}.\,1. 
Let $\{R_0,R_1\}$ be a ``honeycomb system'' 
 adapted to $\U$.
In our case we may assume that each $R_i$ is a closed real $2n$-dimensional \mfd\ with $C^{\infty}$ boundary
$\de R_{i}$ in $U_{i}$, $i=0,1$, \st\ $R_{0}\cup R_{1}=M$ and  $\op{Int}R_{0}\cap\op{Int}R_{1}=\emptyset$, where ``$\op{Int}$'' means the interior. We set $R_{01}=R_{0}\cap R_{1}$, which 
coincides with $\de R_{0}$ 
and is endowed with the orientation of $\de R_{0}$
so that $R_{01}=\de R_{0}=-\de R_{1}$, as oriented \mfd s. 

Suppose $M$ is compact. Then we may assume that $R_{0}$ and $R_{1}$ are compact and we have the integration
\[
\int_{M}:\cE^{2n}(\U)\lra \C\qquad\text{given by}\ \ \int_{M}\z=\int_{R_0}\z_0+\int_{R_1}\z_1+\int_{R_{01}}\z_{01},
\]
which induces 
the integration on $H^{2n}_{D}(\U)$. It is compatible with the integration on the de~Rham cohomology
$H^{2n}_{d}(M)$ via the \iso\ (\ref{isodr}).

In the relative case we proceed as follows. Thus let $S$ be a closed set in $M$. Let $U_{0}=M\ssm S$ and $U_{1}$ a \nbd\ of
$S$ and consider the covering $\U=\{U_{0},U_{1}\}$ of $M$ as before. We have the relative de~Rham cohomology
$H^{r}_{D}(\U,U_{0})$  as in Example \ref{exdRD}.\,1.
Suppose $S$ is compact ($M$ may not be). Let $R_{1}$ be a $2n$-dimensional compact sub\mfd\ of $M$
with $C^{\infty}$ boundary containing $S$ in its interior and set $R_{01}=-\partial R_{1}$. Then we have the integration
\begin{equationth}\label{intdrr}
\int_{M}:H^{2n}_{D}(\U,U_{0})\lra \C\qquad\text{induced by}\ \ \int_{M}\z=\int_{R_1}\z_1+\int_{R_{01}}\z_{01}.
\end{equationth}

Now we consider the Aeppli case. First suppose $M$ is compact. Then the
integration on $H^{n,n}_{\rm A}(M)$ is defined using the identity $H^{n,n}_{\rm A}(M)=H^{2n}_{d}(M)$. The integration on the \v{C}ech-Aeppli cohomology is defined using
Proposition \ref{propCA}.
Namely,
 we have the integration
 \begin{equationth}\label{intA} 
 \int_{M}:H_{\rm A}^{n,n}(\U)\lra\C\qquad\text{induced by}\ \ 
\int_M\xi=\int_{R_0}\xi_0+\int_{R_1}\xi_1
+\int_{R_{01}}(\xi_{01}^{(1)}+\xi_{01}^{(2)}).
\end{equationth}

In the relative case we may use Proposition \ref{proprA}.
Thus if $S$ is compact,
we have the integration
\[
\int_M: H_{\rm A}^{n,n}(\U,U_0)\lra\C\qquad\text{induced by}\ \ 
\int_M\xi=\int_{R_1}\xi_1+\int_{R_{01}}(\xi_{01}^{(1)}+\xi_{01}^{(2)}).
\]

\begin{remark} Noting that there is a canonical morphism $H^{n,q}_{\bar\vartheta}(\U)\ra H^{n+q}_{D}(\U)$,
we may define the integration on $H^{n,n}_{\rm BC}(\U)$ as the composition
\[
H^{n,n}_{\rm BC}(\U)\lra H^{n,n}_{\bar\vartheta}(\U)\lra H^{2n}_{D}(\U)\overset{\int_{M}}\lra\C.
\]
\end{remark}

\subsection{Duality morphisms}

If $M$ is compact, we have a bilinear map
\begin{equationth}\label{ABCpair}
 H^{p,q}_{\rm BC}(M)\times H^{n-p,n-q}_{\rm A}(M)\overset\wedge\lra H^{n,n}_{\rm A}(M)\overset{\int_M}\lra\C,
\end{equationth}
which is non-degenerate  so that we have an \iso\ (cf. \cite{Sch})
\begin{equationth}\label{ABCdual}
\varPhi:H^{p,q}_{\rm BC}(M)\overset{\sim}\lra H^{n-p,n-q}_{\rm A}(M)^*.
\end{equationth}

Now we consider the duality between \v{C}ech-Bott-Chern and \v{C}ech-Aeppli cohomologies.
If $M$ is compact, we have a pairing (cf. Proposition \ref{cupcbcca} and \eqref{intA})\,:
\[
 H_{\rm BC}^{p,q}(\U) \times H_{\rm A}^{n-p,n-q}(\U)\overset\ssl\lra H_{\rm A}^{n,n}(\U)\overset{\int_{M}}\lra  \C,
\]
which is compatible with (\ref{ABCpair}). Thus we have an \iso
\[
 \varPhi:H_{\rm BC}^{p,q}(\U) \overset\sim\lra H_{\rm A}^{n-p,n-q}(\U)^*\simeq H_{\rm A}^{n-p,n-q}(M)^*.
\]

In the relative case, suppose $S$ is compact. 
We then have a pairing (cf. Proposition~\ref{propcuprel} and \eqref{intdrr})\,:
\[
H_{\rm BC}^{p,q}(\U,U_0)\times
 H_{\rm A}^{n-p,n-q}(U_1)\overset\cdot\lra  H_{D}^{2n}(\U,U_0)\overset{\int_M}\lra \C,
\]
which induces  a morphism $H_{\rm BC}^{p,q}(\U,U_0)\ra H_{\rm A}^{n-p,n-q}(U_1)^*$. If we set
\[
H_{\rm A}^{n-p,n-q}[S]=\underset{\underset{U_{1}\supset S}\lra}\lim\, H_{\rm A}^{n-p,n-q}(U_1),
\]
the inductive limit over the open \nbd s $U_{1}$ of $S$, we have a morphism (cf. Corollary \ref{corind})
\begin{equationth}\label{ABC}
\varPsi:H_{\rm BC}^{p,q}(\U,U_0)\lra H_{\rm A}^{n-p,n-q}[S]^*.
\end{equationth}

\begin{proposition}\label{commutativity} If $M$ is compact, we have the following commutative diagram\,:
\[
\SelectTips{cm}{}
\xymatrix
{H_{\rm BC}^{p,q}(\U,U_0)\ar[r]^-{j^{*}} \ar[d]^-{\varPsi}& H_{\rm BC}^{p,q}(M)\ar[d]^-{\varPhi}_-{\wr}\\
H_{\rm A}^{n-p,n-q}[S]^*\ar[r]^-{i_*} & H_{\rm A}^{n-p,n-q}(M)^*.}
\]
\end{proposition}

\lsection{Characteristic classes in  \v Cech-Bott-Chern cohomology}\label{secchBC}

\subsection{Chern forms}

We start by briefly recalling the Chern-Weil theory of
characteristic classes. For details
we refer to \cite{BB}, \cite{BC}, \cite{MS} and \cite{Sw1}.

Let $M$ be a $C^\infty$ \mfd\ and $E\ra M$  a $C^\infty$ complex \vb\ of rank $l$.
For an open set $U$ in  $M$, we denote by
$ \cE^r(U,E)$  the vector space of $E$-valued $r$-forms on $U$, i.e., $C^{\infty}$ sections of the bundle
$\bigwedge^r\tau^*\otimes E$ on $U$, where $\tau=T^c_{\R}M$ denotes
 the complexification of the real tangent bundle $T_{\R}M$ of $M$. 

 A  {\em connection} for $E$  is a ${\C}$-linear map
\[
\na: \cE^0(M,E)\lra  \cE^1(M,E)
\]
which is a derivation\,:
$\na(fs)=df\otimes s+f\na(s)$ for $ f\in \cE^0(M)$, $s\in  \cE^0(M,E)$.
For 
a vector field $u$, we denote $\na(s)(u)$ by $\na_u(s)$, which is in $ \cE^0(M,E)$.
A connection $\na$ induces a derivation
$\na: \cE^1(M,E)\ra  \cE^2(M,E)$ and the composition
\[
K=\na\circ\na: \cE^0(M,E)\lra  \cE^2(M,E)
\]
is the {\em curvature} of $\na$. It is $ \cE^0(M)$-linear so that it can be thought of as a $2$-form with
coefficients in $\op{Hom}(E,E)$. Thus, for the $p$-th elementary invariant polynomial $\s_p$, $p=1,\dots,l$,
$\s_p(K)$ is a $2p$-form on $M$, which is shown to be closed. The {\em $p$-th Chern form} of $\na$ is defined by
\[
c^p(\na)=\Bigl(\frac{\sqrt{-1}}{2\pi}\Bigr)^p\s_p(K).
\]

More generally, if $\varphi$ is an invariant polynomial
homogeneous of degree $k$, we have a well-defined $2k$-form $\varphi(K)$, which is closed. By a slight abuse of
notation, we set
\[
\varphi(\na)=\Bigl(\frac{\sqrt{-1}}{2\pi}\Bigr)^k\varphi(K).
\]
Alternatively, for such $\varphi$,
there is a polynomial $P$ such that $\varphi=P(\s_1,\s_2,\dots)$. Then
$\varphi(\na)=P(c^1(\na),c^2(\na),\dots)$.

Noting that a connection is a local operator, we have local representations of the connection and curvature.  For a (local) frame $\bm{e}=(e_1,\dots,e_l)$ of $E$, the 
connection matrix $\t=(\t_{ij})$ is determined by $\na e_i=\sum_{i=1}^l \t_{ji}e_j$. 
Also the 
curvature matrix $\kappa=(\kappa_{ij})$ is determined by $K e_i=\sum_{i=1}^l \kappa_{ji}e_j$. 
We then have
\begin{equationth}\label{curvconn}
\kappa=d\t+\t\wedge\t.
\end{equationth}
By a frame change $\bm{e}'=\bm{e}A$, the connection and curvature matrices become
\[
\t'=A^{-1}\cdot dA+A^{-1}\t A,\qquad \kappa'=A^{-1}\kappa A.
\]
The form $\varphi(\na)$ above is locally given by $(\frac{\sqrt{-1}}{2\pi})^{k}\varphi(\kappa)$.

\paragraph{Bott difference form\,:}

Let $\varphi$ be as above.
For two connections $\na$ and $\na'$ for $E$,
there exists a $(2k-1)$-form $\varphi(\na,\na')$ \st\ $\varphi(\na',\na)=-\varphi(\na,\na')$
and that
\[
d\hspace{.3mm}\varphi(\na,\na')=\varphi(\na')-\varphi(\na).
\]
As to the construction of such forms, we refer to \cite{B},  \cite{BC} and \cite{Sw1}.
One of the consequences of this  is that  the class of $[\varphi(\na)]$ in $H_{d}^{2k}(M)$ is independent of the choice of the connection $\na$ for $E$. The class will be  denoted by $\varphi(E)$. In particular the class of $c^{p}(\na)$ is the $p$-th Chern class
 $c^{p}(E)$ of $E$.

\subsection{Metric connections}

\paragraph{Connections of type $\bm{(1,0)}$\,:}
Now let $M$ be a complex \mfd\ 
of dimension $n$ 
and $E\ra M$ a holomorphic \vb\ of rank $l$.
 A connection $\na$ for $E$ is  {\em of type $(1,0)$} (a  $(1,0)$-connection for short) if the entries of the connection matrix with respect to a holomorphic frame are forms of type $(1,0)$.
Note that this property does not depend on the choice of  the
holomorphic frame and that every holomorphic \vb\ admits a $(1,0)$-connection.
If $\na$ is a $(1,0)$-connection for $E$, we may write its curvature $K$ as
\[
K=K^{2,0}+K^{1,1}
\]
with $K^{2,0}$ and $K^{1,1}$ a $(2,0)$-form and a $(1,1)$-form with coefficients in ${\rm Hom}(E,E)$. Locally, if $\t$ and $\kappa$ are  the connection and  curvature
matrices of $\na$ \wrt\ a  holomorphic frame, 
we can decompose as $\kappa=\kappa^{2,0}+\kappa^{1,1}$ according to the types, and $K^{2,0}$
and $K^{1,1}$ are  represented by (cf. \eqref{curvconn})
\begin{equationth}\label{curvform}
\kappa^{2,0}=\de\t+\t\wedge\t\qquad\text{and}\qquad
\kappa^{1,1}=\dbar\t.
\end{equationth}%
Thus $K^{1,1}$, being locally $\dbar$-exact, is a $\dbar$-closed $(1,1)$-form with coefficients in 
$\op{Hom}(E,E)$.

\paragraph{Atiyah forms\,:}
In the above situation, for a $(1,0)$-connection $\na$,
we have the $\dbar$-closed 
$(p,p)$-form $\s_p(K^{1,1})$ on $M$, which is locally  given by  $\s_p(\kappa^{1,1})$.
The {\em $p$-th Atiyah form} of $\na$ is defined by (cf. \cite{ABST})
\[
a^p(\na)=\Bigl(\frac{\sqrt{-1}}{2\pi}\Bigr)^p\s_p(K^{1,1}).
\]

 If $\na$ is a $(1,0)$-connection, the $p$-th Chern form $c^p(\na)$ 
 is  a  $2p$-form having components of types $(2p,0),\ldots, (p,p)$. The Atiyah form $a^p(\na)$ is  the $(p,p)$-component of $c^p(\na)$. In particular, $a^n(\na)=c^n(\na)$. 
 
 If $\varphi$ is an invariant  polynomial homogeneous of degree $k$ and if $\na$ is a 
$(1,0)$-connection, $\varphi(\na)$ is  a closed $2k$-form having components of types $(2k,0),\dots, (k,k)$. 
We denote the $(k,k)$-component  of $\varphi(\na)$ by $\varphi^{{\rm A}}(\na)$ and call it
the Atiyah form of $\na$ \wrt\ $\varphi$.
If $\varphi=P(\s_{1},\s_{2},\dots)$, then $\varphi^{{\rm A}}(\na)=P(a^{1}(\na),a^{2}(\na),\dots)$,
which is a $\dbar$-closed $(k,k)$-form on $M$.

If we have two  $(1,0)$-connections $\na$ and $\na'$ for $E$,
we have the difference form $\varphi^{{\rm A}}(\na,\na')$.
It is the $(k,k-1)$-component of $\varphi(\na,\na')$ and satisfies
\[
\dbar \varphi^{{\rm A}}(\na,\na')=\varphi^{{\rm A}}(\na')-\varphi^{{\rm A}}(\na).
\]
Thus,  if $\na$ is a $(1,0)$-connection for $E$, the class of $\varphi^{{\rm A}}(\na)$ in $H^{k,k}_{\dbar}(M)$ does not depend on the choice of $\na$. It is denoted by $\varphi^{{\rm A}}(E)$ and is called the Atiyah class of $E$ \wrt\ $\varphi$. In particular, the class of $a^p(\na)$ in $H^{p,p}_{\dbar}(M)$ is 
the {\em $p$-th Atiyah class} $a^p(E)$ of $E$.

\paragraph{Hermitian connections\,:}

A {\em Hermitian vector bundle} $(E,h)$ is a $C^{\infty}$ complex vector bundle $E$ together with a Hermitian metric $h$ on $E$.
 A connection $\na$ for $E$
is an {\em $h$-connection}  if it is compatible with $h$\,:
\[
d\hspace{.2mm}h(s,t)=h(\na s,t)+h(s,\na t)\qquad\text{for}\ \ s, t\in \cE^0(M,E).
\]
A connection $\na$ for $E$ is  {\em Hermitian},  if it is an $h$-connection for some 
$h$. 
Note that every Hermitian vector bundle $(E,h)$ admits an $h$-connection. Thus every complex vector bundle admits a Hermitian connection. 

\paragraph{Metric connections\,:}

Let $M$ be a complex \mfd\ and $E\ra M$ a holomorphic \vb. Recall that, for any Hermitian metric $h$ on $E$, 
there exists a unique
$h$-connection of type $(1,0)$. We call such a connection the {\em  $h$-metric connection}.
Let $\na$ be the $h$-metric connection. Let $\bm{e}=(e_1,\dots,e_l)$ be a (local) holomorphic frame and $\t$ the connection matrix of $\na$ \wrt\ $\bm{e}$. We also
set $H=(h_{ij})$, $h_{ij}=h(e_i,e_j)$. Then the connection matrix \wrt\  $\bm{e}$ is given by
\begin{equationth}\label{metric}
\t={}^t\hspace{-.7mm}H^{-1}\cdot\de {}^t\hspace{-.7mm} H.
\end{equationth}

We call a Hermitian connection of type $(1,0)$ a {\em metric connection}. 
From \eqref{metric} we see that,
if $\na$ is a metric connection for $E$, its curvature matrix $\kappa$ \wrt\ a holomorphic frame is of type $(1,1)$\,:
\begin{equationth}\label{curvmetric}
\kappa=\kappa^{1,1}=\dbar\t.
\end{equationth}
Thus we have\,:
\begin{proposition}\label{a=c} Let $E$ be a holomorphic  vector bundle. If $\na$ is a metric connection for $E$, the Chern forms and the Atiyah forms are the same\,{\rm :}
\[
\varphi(\na)=\varphi^{{\rm A}}(\na).
\]
Thus they are simultaneously $d$  and $\dbar$-closed.
\end{proposition}

\begin{remark}\label{remm} {\bf 1.} For two metric connections $\na$ and $\na'$, we also have 
\[
\varphi(\na,\na')=\varphi^{{\rm A}}(\na,\na').
\]

\noindent
{\bf 2.} Let $(E,h)$ be a Hermitian line bundle and $s$ a non-vanishing holomorphic section of $E$.  We may think of $s$ as a frame and set $N(s)=H=h(s,s)$, the square of the norm of $s$. 
Then, by (\ref {metric}) and (\ref{curvmetric}), 
the first Chern form of the $h$-metric connection $\na$ is given by
\[
c^1(\na)=\frac {\sqrt{-1}} {2\pi}\kappa=\frac {\sqrt{-1}} {2\pi}\dbard\log N(s),
\]
which is also equal to $a^1(\na)$.
\end{remark}

\subsection{Bott-Chern classes}

Let $E\ra M$ be a \h\ \vb\ on a complex \mfd. For an invariant polynomial $\varphi$ and a Hermitian metric $h$ on $E$ we set $\varphi(E,h)=\varphi(\na)$, where $\na$ is the $h$-metric connection. 
If we have two metrics $h$ and $h'$ on $E$, we have the Bott-Chern  form $\varphi(E,h,h')$
which relates $\varphi(E,h)$ and $\varphi(E,h')$. It was first introduced in \cite{BC} and then has been studied by a number of authors via different approaches (cf. \cite{BGS}, \cite{BKK}, \cite{BL}, \cite{GS1} and references therein).
We may summarize what we need in our situation as\,:

\begin{theorem}\label{thBCd} 
Let $\varphi$ be an invariant polynomial homogeneous of degree $k$. There is a unique way of assigning to $(E,h,h')$ a $(k-1,k-1)$-form 
$\varphi(E,h,h')$ so that
\begin{enumerate}
\item[\rm (1)] $\bp\de\varphi(E,h,h')=\varphi(E,h')-\varphi(E,h)$,
\item[\rm (2)] $f^{*}\varphi(E,h,h')=\varphi(f^{*}(E,h,h'))$ for every \h\ map $f:M'\ra M$,
\item[\rm (3)] $\varphi(E,h,h')=0$, if $h=h'$.
\end{enumerate}
\end{theorem}

In particular from (1) above we see that the class of $\varphi(E,h)$ in $H^{k,k}_{\rm BC}(M)$ does not depend on the
choice of the metric $h$.

\begin{definition}
The {\em Bott-Chern class} $\varphi_{\rm BC}(E)$ of $E$ \wrt\ $\varphi$ is the class  in  $H^{k,k}_{\rm BC}(M)$  of $\varphi(E,h)$ for some Hermitian metric $h$ on $E$.
\end{definition}

In the sequel we denote $\varphi(E,h)$ and $\varphi(E,h,h')$ simply by $\varphi(h)$ and $\varphi(h,h')$, if the bundle under consideration is understood.

\begin{remark}\label{BC=A}
%
{\bf 1.}
The Bott-Chern Chern class and the Bott-Chern Atiyah class are the same\,: 
$c^p_{\rm BC}(E)=a^p_{\rm BC}(E)$ in $H^{p,p}_{\rm BC}(M)$. It goes to the Atiyah class $a^p(E)$ by the first morphism in (\ref{nathomos}) and to the Chern class
$c^p(E)$ by the second.
\smallskip

\noindent
{\bf 2.} Suppose $E$ is a line bundle with Hermitian metrics $h$ and $h'$. For a non-vanishing section $s$,
we set $N(s)=h(s,s)$ and $N'(s)=h'(s,s)$. Then, for $\varphi=c^{1}$, we have
\[
c^{1}(h,h')=\frac{\sqrt{-1}}{2\pi}\log \frac{N'(s)}{N(s)}.
\]
\smallskip

\noindent
{\bf 3.} In general, if $h_{0}$ and $h_{1}$ are Hermitian metrics on $E$ defined on open sets $U_{0}$ and $U_{1}$
in $M$, \r, the form $\varphi(h_{0},h_{1})$ as above is defined on $U_{01}=U_{0}\cap U_{1}$.
\end{remark}

\subsection{Characteristic classes in  \v Cech-Bott-Chern cohomology}

Let $E$ be a \h\ \vb\ on a complex manifold $M$ and $\U=\{U_0,U_1\}$ an open covering
of $M$. For $i=0,1$, let $h_i$ be a  metric  for $E$ on $U_i$. If $\varphi$ is an invariant polynomial
homogeneous of degree $k$, the cochain
\begin{equation}\label{BCcocycle}
\varphi(h_*)=(\varphi(h_0),\varphi(h_1),\varphi(h_0,h_1))
\end{equation}
in
\[
\cE^{k,k}_{\rm BC}(\U)=\cE^{k,k}(U_0)\oplus \cE^{k,k}(U_1)\oplus
\cE^{k-1,k-1}(U_{01}).
\]
is  a cocycle. 
Thus  we have  a class
$[\varphi(h_*)]$ in $ H_{\rm BC}^{k,k}(\U) $.

\begin{proposition}\label{propindep}
The class does not depend on the choice of the metrics. Moreover it corresponds to $\varphi_{BC}(E)\in
H^{k,k}_{\rm BC}(M)$ via the \iso\ \eqref{isobc}.
\end{proposition}

\begin{proof} We recall (cf. \cite[Proposition 4.23]{BKK}) that, for three metrics $(h,h',h'')$, there exist
a $(k-1,k-2)$-form $\varphi^{(1)}(h,h',h'')$ and a $(k-2,k-1)$-form $\varphi^{(2)}(h,h',h'')$ \st
\[
\varphi(h',h'')-\varphi(h,h'')+\varphi(h,h')+\bp\varphi^{(1)}(h,h',h'')+
\de\varphi^{(2)}(h,h',h'')=0.
\]

Let $h_{0}'$ be another metric on $E$ on $U_{0}$ and set 
\[
\varphi(h_{*}')=(\varphi(h_0'),\varphi(h_1),\varphi(h_0',h_1)).
\]
Then we have
\[
\begin{aligned}
\varphi(h_{*}')-\varphi(h_{*})&=(\varphi(h_0')-\varphi(h_0),0,\varphi(h_0',h_1)-\varphi(h_0,h_1))\\
&=(\dbard\varphi(h_{0},h_{0}'),0,-\varphi(h_{0},h_{0}')-\bp\varphi^{(1)}(h_{0},h_{0}',h_{1})-
\de\varphi^{(2)}(h_{0},h_{0}',h_{1}),
\end{aligned}
\]
which is in $B^{k,k}_{\rm BC}(\U)$ (cf. (\ref{CBCcocycob})) so that $[\varphi(h_{*}')]=[\varphi(h_{*})]$.

Likewise  Let $h_{1}'$ be another metric on $E$ on $U_{1}$ and set 
\[
\varphi(h_{*}')=(\varphi(h_0),\varphi(h'_1),\varphi(h_0,h'_1)).
\]
Then we have
\[
\begin{aligned}
\varphi(h_{*}')-\varphi(h_{*})&=(0,\varphi(h_1')-\varphi(h_1),\varphi(h_0,h'_1)-\varphi(h_0,h_1))\\
&=(0,\dbard\varphi(h_{1},h_{1}'),\varphi(h_{1},h_{1}')+\bp\varphi^{(1)}(h_{0},h_{1},h'_{1})+
\de\varphi^{(2)}(h_{0},h_{1},h'_{1}),
\end{aligned}
\]
which is in $B^{k,k}_{\rm BC}(\U)$ (in fact in $B^{k,k}_{\rm BC}(\U,U_{0})$).

The last part can be seen by comparing with the class defined by a global metric.
\end{proof}

\lsection{A vanishing theorem}\label{ABvanish}

\subsection{Actions of  distributions}\label{locdist}

Let $M$ be a complex \mfd\ of dimension $n$ and $F$  a non-singular distribution of dimension $p$, i.e., a  subbundle of $TM$ of  rank $p$.

\begin{definition} A (holomorphic) {\it action} of $F$ on a 
holomorphic \vb\ $E$ over $M$ is a $\C$-bilinear map
\[
\a: \cE^0(M,F)\times  \cE^0(M,E)\lra  \cE^0(M,E)
\]
satisfying the following conditions, for $f\in \cE^0(M)$, $u\in \cE^0(M,F)$,  $s\in \cE^0(M,E)$\,:
\begin{enumerate}
\item[(1)] $\a(fu,s)=f\a(u,s)$,
\item[(2)] $\a(u,fs)=u(f)s+f\a(u,s)$ and
\item[(3)] $\a(u,s)$ is holomorphic whenever $u$ and $s$ are.
\end{enumerate}
\end{definition}

A \vb\ $E$ with an action of $F$ is called an {\it $F$-bundle}.

\begin{definition}\label{defFc} Let $E$ be an $F$-bundle with action $\a$.
An {\em $F$-connection}
for $E$ 
is a $(1,0)$-connection $\na$ with
\[
\na_u(s)=\a(u,s),\qquad\text{for}\ \ s\in \cE^0(M,E),\ u\in \cE^0(M,F).
\]
\end{definition}

From the fact that an action is a local operation, we see that an $F$-bundle always admits an $F$-connection.

We note that the above material can be equivalently treated in 
terms of partial holomorphic connections instead of actions (cf.  \cite{ABST}), the  condition (4) in Remark 
\ref{remv}.\,1 below corresponding to the fact that the partial connection is flat.
We have the following Bott type vanishing theorem for $F$-connections, which is proved in \cite[Theorem 6.10]{ABST} in the context of partial connections.

\begin{theorem}
Let $M$ and $F$ be as above.
Let $E$ be an $F$-bundle and $\na$  an $F$-connection for $E$.
For  an 
invariant polynomial $\varphi$ homogeneous of degree $k>n-p$, we have\,{\rm :}
\[
\varphi^{\rm A}(\na)= 0.
\]\label{ABvanishing}
\end{theorem}

\begin{corollary}\label{vanherm} If, in addition, $\na$ is Hermitian,
\[
\varphi(\na)= 0.
\]
\end{corollary}

\begin{remark}\label{remv} {\bf 1.} For two connections $\na$ and $\na'$ and an invariant polynomial $\varphi$ as in Theorem  \ref{ABvanishing},
we have the vanishing $\varphi^{\rm A}(\na,\na')= 0$. Thus if, in addition, $\na$ and $\na'$ are
Hermitian, $\varphi(\na,\na')= 0$ (cf. Remark \ref{remm}.\,1). However if $\na$ and $\na'$ are
Hermitian \wrt\ different metrics $h$ and $h'$, the form
$\varphi(h,h')$ does not vanish in general, as the following example, which was communicated to us by  the referee, shows.

Let $E=M\times\C$ be the product bundle and $F=TM$ so that $p=n$. Then there is a natural action of $F$ on $E$
(the action of vector fields on \fcn s). The exterior derivative $d$ is an $F$-connection for $E$.
 It is also the metric
connection \wrt\ any constant metric on $E$, i.e., a metric
 $h$ 
given by $h(e,e)=a$ (a positive real number) with
$e$ the frame of $E$ defined by $x\mapsto (x,1)$. Note that for $\varphi=c^{1}$, $\deg \varphi=1>n-p=0$. However,
for the two metrics $h$ and $h'$ 
given by $h(e,e)=1$ and $h'(e,e)=a\ne1$, we have (cf. Remark~\ref{BC=A}.\,3)
\[
c^{1}(h,h')=\frac{\sqrt{-1}}{2\pi}\log a\ne 0.
\]
\smallskip

\noindent
{\bf 2.}
In Theorem \ref{ABvanishing}, if $F$ is involutive and if the action satisfies
\begin{enumerate}
\item[(4)] $\a([u,v],s)=\a(u,\a(v,s))-\a(v,\a(u,s))$,
\end{enumerate}
we have $\varphi(\na)= 0$ for $\varphi$ with $k> n-p $. This is usually referred to as the Bott vanishing theorem.
\smallskip

\noindent
{\bf 3.} Let $h$ be a Hermitian metric on $E$. Under the condition of Corollary \ref{vanherm}, we have 
\[
u(h(s,t))=h(\na_u(s),t),
\]
for $u$ in $A^0(M,F)$ and a holomorphic section $t$ of $E$.
\smallskip

\noindent
{\bf 4.} Some special cases of this Bott type vanishing theorem  are
proved in \cite{Carrell}.
\end{remark}

\subsection{Action on the normal bundle of an invariant sub\mfd}\label{sscs} 

Let $M$ be a complex manifold of dimension $n$ and $V$ a complex sub\mfd\ of dimension $d$ of $M$. 
Let $N_V$ be the normal bundle of $V$ in $M$ so that we have the exact sequence
\begin{equationth}\label{exactnormal}
0\lra TV\lra TM|_V\overset\pi\lra N_V\lra 0.
\end{equationth}
Let $F$ be a distribution of dimension $p$ on $M$. 
We say that $F$ leaves $V$ invariant if $F|_V\subset TV$. In this case we set $F_{V}=F|_{V}$, which is a 
a distribution of dimension $p$ on $V$.
The following is proved in \cite{Leh3} (see also \cite{LS}) for the case of foliations. 
In fact the involutivity of $F$ is not necessary and a proof is given in \cite{ABST} in terms of partial connections
and in \cite{Sw3} in terms of actions. 

\begin{theorem}\label{holo-normal}
In the above situation,
there is a natural holomorphic action  of $F_V$ on 
$N_V$.
\end{theorem}

The action is constructed as follows. Let $u$ and $\nu$ be $C^\infty$ sections of $F_V$ and $N_V$, respectively. Take sections $\tilde u$ of $F$ and $\tilde v$ of $TM$ so that $\tilde u|_V=u$ and $\pi(\tilde v|_V)=\nu$, where $|_V$ means the restriction as  sections. Define 
\begin{equationth}\label{csaction}
\a:\cE^0(V,F_V)\times \cE^0(V,N_V)\lra \cE^0(V,N_V)\qquad\text{by}\ \ \a(u,\nu)=\pi([\tilde u,\tilde v]|_V).
\end{equationth}
Then it is a well-defined action, which is referred to as the {\em Camacho-Sad action}.

\begin{corollary} In the above situation, let $\na$ be an $F_V$-connection for $N_V$, which is also Hermitian
\wrt\ some Hermitian metric on $N_V$. Then, for an invariant  polynomial $\varphi$ homogeneous of degree $k> d-p$, we have $\varphi(\na)= 0$.
\end{corollary}

\lsection{Localization and Hermitian residues}\label{secloc}

We briefly recall singular \h\ distributions, for which we refer to \cite{Sw3} for details.
Then we prove a residue theorem for 
 vector bundles admitting a Hermitian connection compatible with an action of the non-singular part of a singular distribution.

\paragraph{Singular \h\ distributions\,:}  Let $M$ be a complex \mfd\ of dimension $n$.  We denote by
${\cal O}_{M}$ and $\varTheta_{M}$ its structure sheaf and the tangent sheaf. For simplicity
we assume that $M$ is connected. {\em A singular distribution}
on $M$ is a coherent subsheaf  $\F$ of $\varTheta_{M}$. We set
\[
S(\F)=\{\,x\in M\mid (\varTheta_{M}/\F)_{x}\ \text{is not}\ {\cal O}_{M,x}\text{-free}\,\}
\]
and call it the singular set of $\F$. Away from $S(\F)$, $\F$ is a locally free $\O$-module and its rank is
called the dimension of $\F$. For instance, if $\F$ is generated by a single vector field $v$, $S(\F)$ is the
set of zeros of $v$.

If $\F$ is involutive, i.e., if $[\F,\F]\subset\F$, then $\F$ is called a {\em singular foliation}.

\paragraph{The residue theorem\,:}
Let $\F$ be a singular distribution of dimension $p$ on $M$  with singular set $S=S(\F)$. There is a rank $p$ subbundle $F_0$ of $TM|_{U_0}$, $U_0=M\moins S$, so that $\F|_{U_0}=\O_{U_0}(F_0)$, the sheaf of \h\ sections of $F_{0}$. Letting  $U_1$ be a
neighborhood of $S$, we consider the
covering $\U=\{U_0,U_1\}$  of $M$.  Let $E$ be a holomorphic \vb\ on $M$ with an action of $F_0$ on $U_0$. 
Let $\na_0$ be an $F_0$-connection for $E$ on $U_0$ and suppose there exists a Hermitian metric $h_{0}$ of $E$ 
on $U_{0}$ \st\ $\na_{0}$ is also an $h_{0}$-connection.
Take a Hermitian metric $h_1$ of $E$ on $U_1$ and 
let $\na_1$ be the $h_1$-metric connection for $E$ on 
$U_1$. Recall that for an 
invariant polynomial $\varphi$ homogeneous of degree $k$, the characteristic
class $\varphi_{\rm BC}(E)$ in $H^{k,k}_{\rm BC}(M)\simeq H^{k,k}_{\rm BC}(\U)$ is represented by the cocycle $\varphi(h_*)$
in $\cE_{\rm BC}^{k,k}(\U)$ given by (\ref{BCcocycle}).  If $k>n-p$, then by 
Corollary~\ref{vanherm}, $\varphi(h_{0})=\varphi(\na_{0})= 0$ and $\varphi(h_*)$ is expressed as
\[
\varphi(h_*)=(0,\varphi(h_1),\varphi(h_{0,}h_1))
\]
so that it is in $\cE_{\rm BC}^{k,k}(\U,U_0)$. Thus it  defines
a class in $H^{k,k}_{\rm BC}(\U,U_0)$, which we denote by $\varphi_{\rm BC}(E;\F)$ and call the localization of $\varphi_{\rm BC}(E)$ by $\F$ at $S$. It is sent to the  class $\varphi_{\rm BC}(E)$ by
the canonical morphism
\[
j^*:  H_{\rm BC}^{k,k}(\U,U_0)\lra H_{\rm BC}^{k,k}(M).
\]

\begin{remark} {\bf 1.} The localization $\varphi_{\rm BC}(E;\F)$ above depends a priori  on the metric $h_{0}$. Let $h_{0}'$ be
another metric of $E$ on $U_{0}$ and $\na_{0}'$ the $h_{0}'$-$F_{0}$-connection for $E$ on $U_{0}$. For the usual proof of the independence of the localization (cf. \cite[Ch.\,III, Lemma 3.1]{Sw1}), we need to have the vanishing of
$\varphi(h_{0},h'_{0})$, however this is not the case in general (cf. Remark~\ref{remv}.\,1).
On the contrary the localization is independent of the choice of the metric $h_{1}$ on $U_{1}$, for a fixed $h_{0}$ (cf. the second half of the proof of Proposition \ref{propindep}).
\smallskip

\noindent
{\bf 2.} In the above situation we have the ``Atiyah localization''  $\varphi^{\rm A}(E;\F)$ in 
$H^{k,k}_{\bar\vt}(\U,U_{0})$ (cf. \cite{ABST}). It is represented by the cocycle 
$(0,\varphi^{\rm A}(\na_{1}),\varphi^{\rm A}(\na_{0},\na_{1}))$, which is equal to $(0,\varphi(\na_{1}),\varphi(\na_{0},\na_{1}))$ in this case (cf. Proposition \ref{a=c} and Remark \ref{remm} 1). 
\end{remark}

Suppose $S$ is compact. Then the  image of $\varphi_{\rm BC}(E;\F)$ by the  morphism (cf. (\ref{ABC}))
\[
\varPsi:H^{k,k}_{\rm BC}(\U,U_0)\lra H^{n-k,n-k}_{\rm A}[S]^*
\]
is denoted by $\op{Res}_{\varphi_{\rm BC}}(\F,E;S)$ and called the residue of $\F$ for $E$ at $S$ with respect to $\varphi$.
If $S$ has a finite number of connected components $(S_\l)$, we take an open \nbd\ $U_\l$ of $S_\l$ in $U_1$ for each $\l$ so that $U_\l\cap U_\mu=\emptyset$ if $\l\ne\mu$. Then 
we have the residue $\op{Res}_{\varphi_{\rm BC}}(\F,E;S_\l)$ in $H^{n-k,n-k}_{\rm A}[S_\l]^*$
 for each $\l$, $H_{\rm A}^{n-p,n-q}[S_{\l}]=\underset{\underset{U_{\l}\supset S_{\l}}\lra}\lim\, H_{\rm A}^{n-p,n-q}(U_\l)$. 

Let $R_\l$ be a $2n$-dimensional manifold with
$C^{\infty}$ boundary in $U_\l$ containing  $S_\l$ in
its interior and set $R_{0\l}=-\de R_\l$.
Then the residue $\op{Res}_{\varphi_{\rm BC}}(\F,E;S_\l)$ is
represented by a functional (cf. (\ref{cup3}) and (\ref{intdrr}))
\begin{equationth}\label{reshom}
\xi\mapsto\int_{R_\l}\varphi(h_1)\wedge \xi+\int_{R_{0\l}}\big(
((\de-\bp)\varphi(h_{0,}h_1))\wedge \xi  -\frac{1}{2}(\de-\bp) (\varphi(h_{0,}h_1)\wedge \xi)\big)
\end{equationth}
for every $\dbarde$-closed $(n-k,n-k)$-form $\xi$ in a \nbd\ of $S_{\l}$.

From the above considerations and Proposition \ref{commutativity}, we have
the following\,:

\begin{theorem}\label{residueth} Let $M$, $\F$ and $S$ be as above.
Suppose $S$ is a compact
set  with a finite number of connected components $(S_\l)_\l$. 
Let $E$ be a holomorphic vector bundle on  $M$. 
Assume that there is an action of $F_0$ on $E$ 
and that there is a Hermitian $F_0$-connection for $E$ on $U_0$.  
Let $\varphi$ be  an invariant
polynomial homogeneous  of degree $k>n-p$.  Then\,{\rm :}
\begin{enumerate}
\item[{\rm (1)}] for each $\l$ we have the residue $\op{Res}_{\varphi_{\rm BC}}(\F,E;S_\l)$
in   $H_{\rm A}^{n-k,n-k}[S_\l]^*$, which is
represented by the functional $(\ref{reshom})$,
\item[{\rm (2)}] if moreover $M$ is compact, 
\[
\sum_\l (i_\l)_*\op{Res}_{\varphi_{\rm BC}}(\F,E;S_\l)=\varPhi(\varphi_{\rm BC}(E))
\qquad\text{in}\ \ H_{\rm A}^{n-k,n-k}(M)^*,
\]
where $i_\l: S_\l \hookrightarrow M$ denotes the
inclusion.
\end{enumerate}
\end{theorem}

A residue as $\op{Res}_{\varphi_{\rm BC}}(\F,E;S_\l)$ is referred to as a {\em Hermitian residue}.

\begin{remark}\label{resnum}
If $k=n$ and  if $M$ is compact and connected,  $H_{\rm A}^{n-k,n-k}(M)^*=H_{\rm A}^{0,0}(M)^*$ may be identified with $\C$, and in this case, 
$ (i_\l)_*{\rm Res}_{\varphi_{\rm BC}}(\F,E;S_\l)$ is a complex number given by
\[
\int_{R_\l}\varphi(h_1)+ \frac{1}{2}\int_{R_{0\l}}
(\de-\bp)\varphi(h_{0,}h_1),
\]
and $\varPhi(\varphi_{\rm BC}(E))$ may be expressed as $\int_M \varphi_{\rm BC}(E)$.
\end{remark}

\lsection{An example}\label{secex}

For $\l \in \C^{*}$ with $|\l|< 1$, we consider the
 Hopf surface $V = (\C^2 \ssm \{0 \})/\sim$, where
$(x,y) \sim  (\l x , \l y)$.
There is a
fibration $\rho: V \ra  \P^1 $ by elliptic  curves. Let $L$ denote the pull-back by $\rho$ of the hyperplane
bundle on $\P^{1}$. Then, as $H^{2}_{d}(V)=0$ and $H^{1,1}_{\bp}(V)=0$,   both the Chern class $c^{1}(L)$  and the Atiyah class $a^{1}(L)$ 
vanish. On the other hand, $H^{1,1}_{\rm BC}(V)\simeq\C$ (cf. \cite{ADT}) and the Bott-Chern class $c^{1}_{\rm BC}(L)$ is a generator. We show that it is
localized at one of the fibers $C$ of $\rho$. For this we realize $V$ as an invariant sub\mfd\ of a singular foliation on an ambient
\mfd\ whose singular set on $V$ is $C$. Then the Camacho-Sad action of the foliation on $L$  away from $C$ gives the
localization (cf. Subsection \ref{sscs}).

We consider the
 Hopf manifold $M = (\C^3 \ssm \{0 \})/\sim$, where $(x,y,z) \sim  (\l x , \l y, \l z)$.
Let $\F$ be the two-dimensional foliation on $M$ induced by the vector fields
\[
v_1=  y\frac{\de}{\de y}+z\frac{\de}{\de z}\quad\text{and}\quad v_2=\frac{\de}{\de x},
\]
 which   leaves invariant the Hopf
surface $V=\{z=0\}/\sim\ \  \subset M $. The singular set $S(\F)$ of $\F$ is given by
\[
S(\F)=C=\{y=z=0\}/\sim\ \  \subset V.
\]
There is a
fibration $\tilde\rho: M \ra  \P^2 $ by elliptic  curves. It restricts to the fibration $\rho:V\ra\P^{1}$, of which  $C$ is a fiber. 
Recall  the exact sequence (\ref{exactnormal})\,:
\[
0\lra TV\lra TM|_{V}\overset\pi\lra N_{V}\lra 0.
\]
If we denote by $\tilde L$ the pull-back of the hyperplane bundle on $\P^{2}$ by $\tilde\rho$, we have $L=\tilde L|_{V}$.
The normal bundle of $V$ is given by $N_V=L$, since $TM=\tilde L\oplus
\tilde L \oplus  \tilde L$ and   $TV=L\oplus L$. 

We try to localize $c^{1}_{\rm BC}(N_V)$ on  $C$.  
Denoting by $[x,y]$ the image of $(x,y)$ by the canonical
surjection $\C^{2}\ssm\{0\}\ra V$,
we set $U_{0}=\{\,[x,y]\mid y\ne 0\}$ and $U_{1}=\{\,[x,y]\mid x\ne 0\}$ and
consider  the covering $\U=\{U_{0},U_{1}\}$ of $V$. We
have that $V\ssm C=U_{0}$ and $C\subset U_{1}$. 
 The bundle $L$ is described as 
\[
L=(U_{0}\times \C)\cup (U_{1}\times \C),
\]
where $([x,y],\z_{0})$ and $([x,y],\z_{1})$ are identified \iff\  $\z_{0}=x/y\cdot\z_{1}$.

Also, if we set $s=x/y$ on $U_{0}$, $s$ is a base coordinate and $y$ is a (covering) fiber coordinate of  the fibration $\rho:V\ra\P^{1}$ and,  
 if we set $t=y/x$ on $U_{1}$, $t$ is a base coordinate and $x$ is a fiber coordinate of  the fibration $\rho$.

The foliation $\F$ defines a subbundle $F_{0}$ of $TV|_{U_{0}}$ of rank $2$ and there is the Camacho-Sad action
of $F_{0}$ on $N_{V}|_{U_{0}}$. Let $\na$ be an $F_{0}$-connection for $N_{V}|_{U_{0}}$ and we compute
the connection form for $\na$ (cf. Definition \ref{defFc} and (\ref{csaction})). The connection is uniquely
determined in our case and turns out to be Hermitian.

We  compute the connection form locally and find that the expression is valid globally on $U_{0}$.
At each point of $U_{0}$, $(x,y,z)$ is a coordinate system on $M$ in a \nbd\ $\tilde W$ of the point. We set
$W=V\cap\tilde W$.
We take  $\nu=\pi(\de /\de z)$ as a \h\ frame  of  $N_V$  on $W$ and
let $\t$ be the connection  form of $\na$ with respect to  $\nu$. Since $\t$ is of type $(1,0)$, we may write
$\t=fdx+gdy$.  Thus $\na_{v_1}(\nu)=yg \nu$ and $\na_{v_2}(\nu)=f \nu$.
On the other hand,
we compute
\[
\na_{v_1}(\nu)= \pi\Big(\Big[  y \frac\de  {\de y} +z\frac\de{\de z},  \frac\de{\de z}\Big]\Big|_{W}\Big)=- \nu
\quad\text{and}\quad 
\na_{v_2}(\nu)=\pi\Big(\Big[\frac\de{\de x}, \frac\de{\de z}\Big]\Big|_{W}\Big)=0.
\]
Thus we conclude that
\[
\t=-\frac{dy}{y}=   -\de \log  |y|^2. 
\]
Note that the expression is a priori on $W$, however it is valid on the whole $U_{0}$. Moreover, this shows that $\na$ is Hermitian. In fact,  on 
 $U_0$, $[x,y]\mapsto \z_{0}=1/y$ is a frame defining a connection form $-\de \log  |y|^2$
 \wrt\ the standard Hermitian metric on $L|_{U_0}\simeq U_0\times \C$ (cf. (\ref{metric}), Remark \ref{remm}.\,2).
 In this case, we can directly verify the vanishing theorem (Corollary \ref{vanherm}) as
 \[
 c^{1}(\na)=\frac {\sqrt{-1}}{2\pi}\bp\t=-\frac {\sqrt{-1}}{2\pi}\dbarde \log  |y|^2=0.
 \]

Now, also we take on  $U_1$
the 
standard metric   and consider  $1/x$ as a
 non-vanishing section  of $N_V=L$ on  $U_1$. Then, setting $\g=\frac 1{2\pi\sqrt{-1}}$, the class $c^{1}_{\rm BC}(N_V)$ is
 represented by the cocycle (cf. Remark \ref{BC=A}.\,3)
\[
\g\hspace{.5mm}(\dbard
\log |y|^2, \dbard\log |x|^2, \log |x/y|^2)=(0,0, \g\log |x/y|^2)=-(0,0, \g\log |t|^2)
\]
and it defines the localization $c^{1}_{\rm BC}(L;\F) \in H^{1,1}_{\rm BC}(\U,U_{0})$ of $c^{1}_{\rm BC}(L)$.

Now we examine its residue $\op{Res}_{c_{\rm BC}^1}(\F, L;C)\in H_{\rm A}^{1,1}[C]^*$.
We have the following diagram, of which the square part is commutative (cf. Proposition \ref{commutativity})\,:
\[
\SelectTips{cm}{}
\xymatrix
{{}&H_{\rm BC}^{1,1}(\U,U_0)\ar[r]^-{j^{*}} \ar[d]^-{\varPsi}& H_{\rm BC}^{1,1}(V)\ar[d]^-{\varPhi}_-{\wr}\\
H_{\rm A}^{1,1}(C)^*\ar[r]^-{\iota_*}&H_{\rm A}^{1,1}[C]^*\ar[r]^-{i_*} & H_{\rm A}^{1,1}(V)^*,}
\]
where $\iota_{*}$ is the transpose of the restriction $\iota^{*}:H_{\rm A}^{1,1}[C]\ra H_{\rm A}^{1,1}(C)$. We have
$H_{\rm A}^{1,1}(C)\simeq\C$ and the composition $\iota^{*}\circ i^{*}:H_{\rm A}^{1,1}(V)\ra H_{\rm A}^{1,1}(C)$ is an \iso, however note that $H_{\rm A}^{1,1}[C]$ is infinite dimensional.

The residue $\op{Res}_{c_{\rm BC}^1}(\F, L;C)$
 is the functional that assigns to each $\dbard$-closed $(1,1)$-form $\xi$ in a \nbd\ of $C$ the value
(cf. (\ref{reshom}))
\begin{equationth}\label{res}
 -\g\int_{R_{01}}\big((\de\log  t
- \bp \log\bar t)\wedge\xi- \frac{1}{2}(\de-\bp)(\log |t|^2\cdot \xi ) \big),
\end{equationth}
where  $R_1=\{\,(t,x)\in U_{1}\mid |t|\leq \delta\,\}$  for some $\delta>0$ and  $R_{01}=-\de R_{1}$.

The canonical generator of $H_{\rm A}^{1,1}(C)=\C$ is given by
\[
\xi_{0}=\frac {\g}{\log |\l|}\hspace{.7mm} d\log x\wedge d\log\bar x.
\]
It may be thought of as representing the class in $H_{\rm A}^{1,1}(V)$ dual to $c^{1}_{\rm BC}(L)$, which is represented by $\g\dbarde\log(|x|^{2}+|y|^{2})$,
in $H^{1,1}_{\rm BC}(V)$.
Now we calculate ${\rm Res}_{c^{1}_{\rm BC}}(\F,L;C)(\xi_{0})$ (cf. \eqref{res}) and verify that it is equal to one. Since $\de\xi_0=\bar \partial \xi_0=0$, 
$\de (\log |t|^2\cdot \xi_0)=(\de \log t)  \wedge \xi_0$ and $\bar\partial (\log |t|^2  \cdot \xi_0)=\bar\partial (\log \bar t)  \wedge \xi_0$. Thus
\[
\op{Res}_{c^{1}_{\rm BC}}(\F,L;C)(\xi_0)
= \frac{\g}{2}\int_{\de R_{1}}(\de \log t -  \bar\partial\log \bar t) \wedge
\xi_{0}
=\int_{C}\xi_{0}=1.
\]

\vv

M. Corr\^ea

Departamento de Matem\'atica 

 Universidade Federal de Minas Gerais

Av. Ant\^onio Carlos 6627

30161-970 Belo Horizonte, Brazil

mauricio@mat.ufmg.br
\vv

T. Suwa

Department of Mathematics 

Hokkaido University

Sapporo 060-0810, Japan

tsuwa@sci.hokudai.ac.jp

\end{document}